\documentclass{amsart}
\usepackage{amscd}
\usepackage{amssymb}
\usepackage[all, knot]{xy}
\usepackage{extpfeil}
\xyoption{all}
\xyoption{arc}
\usepackage{hyperref}

\newcommand{\floor}[1]{{\left\lfloor #1 \right\rfloor}}

\def\ev{\mathrm{ev}}
\def\odd{\mathrm{odd}}

\def\cAt{At^{\text{classical}}}
\def\vT{\underline{T}}
\def\ctop{c^{\text{top}}}
\def\naive{\text{naive}}
\def\HHn{\operatorname{HH}^\text{ad hoc}}

\def\vf{{\mathbf f}}

\def\sst{({\scriptstyle \frac{n}{2}  })}

\def\hQ{\hat{Q}}
\def\hOmega{\hat{\Omega}}
\def\hcJ{\hat{\cJ}}

\def\cPV{\operatorname{ch}^{\text{PV}}}

\def\tr{\operatorname{tr}}

\def\At{\operatorname{At}}

\def\res{\operatorname{res}}
\def\EndMF{{\mathcal E}nd_{mf}}
\def\bnabla{{\mathbb \nabla}}
\def\can{\operatorname{can}}

\def\cEnd{{\mathcal E}nd}
\def\va{{\underline{a}}}

\def\adi#1{{\tiny \left[{\frac{1}{#1}}\right]}}
\def\Proj{\operatorname{Proj}}
\def\chr{\operatorname{char}}
\def\vectwo#1#2{\begin{bmatrix} #1 \\ #2 \end{bmatrix}}

\def\vecfour#1#2#3#4{\begin{bmatrix} #1 \\ #2 \\#3 \\#4\end{bmatrix}}
\def\id{\operatorname{id}}

\def\vx{\underline{x}}

\def\rank{\operatorname{rank}}

\def\pd#1#2{\frac{\partial #1}{\partial #2}}

\def\sTor{\operatorname{\widehat{Tor}}}
\def\sExt{\operatorname{\widehat{Ext}}}

\def\Ext{\operatorname{Ext}}
\def\Tor{\operatorname{Tor}}

\def\cE{\mathcal E}

\def\cC{\mathcal C}
\def\cF{\mathcal F}

\def\cH{\mathcal H}

\def\cP{\mathcal P}
\def\ctP{\tilde{\mathcal P}}
\def\cO{\mathcal O}
\def\cI{\mathcal I}
\def\cJ{\mathcal J}

\def\cL{\mathcal L}

\def\im{\operatorname{im}}
\def\coker{\operatorname{coker}}
\def\ker{\operatorname{ker}}

\def\len{\operatorname{length}}

\def\smsh{\wedge}

\def\fm{\mathfrak m}

\def\dm{\operatorname{dim}}
\def\voplus#1#2{\overset{#1}{\underset{#2}{\oplus}}}

\def\fold{\operatorname{Fold}}
\def\unfold{\operatorname{Unfold}}
\def\Fold{\fold}
\def\Spec{\operatorname{Spec}}
\def\Sing{\operatorname{Sing}}
\def\Nonreg{\operatorname{Nonreg}}

\def\fp{\mathfrak p}

\def\Dsg{\operatorname{D_{sg}}}
\def\Dsing{\Dsg}

\def\Perf{\operatorname{Perf}}

\def\coker{\operatorname{coker}}
\def\Hom{\operatorname{Hom}}
\def\End{\operatorname{End}}
\def\Mat{\operatorname{Mat}}

\newcommand{\Z}{\mathbb{Z}}
\newcommand{\bP}{\mathbb{P}}

\newcommand{\bA}{\mathbb{A}}
\newcommand{\bH}{{\mathbb{H}}}
\newcommand{\bG}{\mathbb{G}}
\def\R{\bR}

\def\A{\bA}
\newcommand{\bR}{{\mathbb{R}}}

\newcommand{\bL}{{\mathbb{L}}}
\newcommand{\bF}{\mathbb{F}}
\newcommand{\bE}{\mathbb{E}}

\def\lra{\longrightarrow}
\def\into{\hookrightarrow}
\def\onto{\xora{}}
\def\and{\, \text{ and } \,}

\newcommand\HomMF{\operatorname{{\mathcal{H}om}_{mf}}}

\newcommand\cHom{\operatorname{{\mathcal{H}om}}}

\newcommand{\xra}[1]{\xrightarrow{#1}}
\newcommand{\sxra}[1]{\xrightarrow{{\scriptscriptstyle #1}}}
\newcommand{\xla}[1]{\xleftarrow{#1}}
\newcommand{\xora}[1]{\xtwoheadrightarrow{#1}}

\def\darrow#1#2{\xtofrom[#2]{#1}}

\numberwithin{equation}{section}
\theoremstyle{plain} 
\newtheorem{thm}[equation]{Theorem}

\newtheorem*{introthm*}{Theorem}

\newtheorem{cor}[equation]{Corollary}
\newtheorem{assumptions}[equation]{Assumptions}
\newtheorem{lem}[equation]{Lemma}
\newtheorem{prop}[equation]{Proposition}

\newtheorem*{ack}{Acknowledgements}

\theoremstyle{definition}
\newtheorem{defn}[equation]{Definition}

\newtheorem{ex}[equation]{Example}

\theoremstyle{remark}
\newtheorem{rem}[equation]{Remark}

\def\ok{\overline{k}}
\def\uTor{\operatorname{\underline{Tor}}}
\def\res{\operatorname{res}}

\def\lots{ {\text{ lower order terms}}}

\begin{document}

\title[Chern Characters for Twisted Matrix Factorizations]{Chern Characters for Twisted Matrix Factorizations and the Vanishing of the Higher Herbrand Difference}

\author{Mark E. Walker}
\address{Department of Mathematics \\
University of Nebraska\\
Lincoln, NE 68588
}
\email{mwalker5@math.unl.edu}
\date{\today}

\begin{abstract} 
We develop a theory of ``ad hoc'' Chern characters for twisted matrix factorizations associated to a scheme $X$, a line bundle $\cL$,
and a regular global section $W \in \Gamma(X, \cL)$. 

As an application, we establish the vanishing, in certain cases, of $h_c^R(M,N)$, the higher Herbrand difference,
and, $\eta_c^R(M,N)$, the higher codimensional analogue of Hochster's theta pairing, where $R$ is a complete intersection of codimension $c$ with isolated
singularities and $M$ and $N$ are
finitely generated $R$-modules. Specifically, we prove such vanishing if $R = Q/(f_1, \dots, f_c)$ has only isolated singularities, $Q$ is a smooth $k$-algebra,
$k$ is a field of characteristic $0$, the $f_i$'s form a regular sequence, and $c \geq  2$.

Such vanishing was previously established in the general characteristic, but graded, setting in \cite{MPSW2}. 
\end{abstract}


\maketitle

\tableofcontents

\section{Introduction}

This paper concerns invariants of 
finitely generated modules over an affine complete intersection ring $R$. Such a ring is one that can be written as $R = Q/(f_1,
\dots, f_c)$, where 
$Q$ is a regular ring and $f_1,
\dots, f_c \in Q$ form a regular sequence of elements. For technical reasons, we will assume $Q$ is a smooth algebra over a field $k$. 
More precisely, this paper concerns invariants of the {\em singularity category}, $\Dsg(R)$,  of $R$. This is the triangulated category 
defined as the Verdier quotient of $D^b(R)$, the bounded derived category of $R$-modules, by the
sub-category of perfect complexes of $R$-modules. More intuitively, the singularity category keeps track of just the ``infinite tail ends'' of projective
resolutions of finitely generated modules. Since $R$ is Gorenstein, $\Dsg(R)$ is equivalent to $\underline{MCM}(R)$, the stable category
of maximal Cohen-Macaulay modules over $R$.

Thanks to a Theorem of Orlov \cite[2.1]{OrlovLGEquivalence}, the singularity category of $R$ is equivalent to the singularity category of the hypersurface 
$Y$ 
of $\bP^{c-1}_Q =
\Proj Q[T_1, \dots, T_c]$ cut out by the element $W= \sum_i f_i T_i \in \Gamma(\bP^{c-1}_Q, \cO(1))$. Since the scheme $\bP^{c-1}_Q$ is regular 
$\Dsg(Y)$ may in turn be given by the homotopy category of ``twisted matrix factorizations''
associated to the triple $(\bP^{c-1}_Q, \cO(1), W)$ (see \cite{PVStack}, \cite{LinPom}, \cite{OrlovLG}, \cite{PositCoherent}, \cite{BW2}).
In general, if $X$ is a scheme (typically smooth over a base field $k$), $\cL$ is a line bundle
on $X$, and $W \in \Gamma(X, \cL)$ is regular global section of $\cL$, a {\em twisted matrix factorization} for $(X, \cL, W)$ consists of a pair of locally free
coherent sheaves $\cE_0, \cE_1$ on $X$ and morphisms $d_1: \cE_1 \to \cE_0, d_0: \cE_0 \to \cE_1 \otimes_{\cO_X} \cL$ such that each composite
$d_0 \circ d_1$ and $(d_1 \otimes \id_{\cL}) \circ d_0$ is multiplication by $W$. In the special case where $X = \Spec(Q)$ for a local ring $Q$, so that $\cL$
is necessarily the trivial line bundle and $W$ is an element of $Q$, a twisted matrix factorization is just a classical matrix factorization as defined by
Eisenbud \cite{Eisenbud}.

Putting these facts together, we get an equivalence
$$
\Dsg(R) \cong hmf(\bP^{c-1}_Q, \cO(1), \sum_i f_i T_i)
$$
between the singularity category of $R$ and the homotopy category of twisted matrix factorizations for $(\bP^{c-1}_Q, \cO(1), \sum_i f_i T_i)$.
The invariants we attach to objects of $\Dsg(R)$ for a complete intersection ring $R$ are actually defined, more generally, in terms of such twisted matrix
factorizations. In detail, 
for  a triple $(X, \cL, W)$ as above, let us assume also that $X$ 
is smooth over a base scheme $S$ and $\cL$ is pulled-back from $S$. In the main case of interest, namely $X = \bP^{c-1}_Q$ for a smooth $k$-algebra $Q$, 
the base scheme $S$ is taken to be
$\bP^{c-1}_k$. In this situation,  
we define a certain {\em ad hoc Hochschild homology} group
$\HHn_0(X/S, \cL, W)$ associated to $(X, \cL, W)$.
Moreover, to each twisted matrix factorization  $\bE$ for $(X, \cL, W)$ we attach a class
$$
ch(\bE) \in \HHn_0(X/S, \cL, W),
$$
called the {\em ad hoc Chern character} of $\bE$.
We use the term ``ad hoc'' since we offer no justification that these invariants are the ``correct'' objects with these names given by a more theoretical framework. 
But, if $X$ is affine and $\cL$ is the trivial bundle, then our ad hoc Hochschild homology groups and Chern characters coincide with the usual notions found in 
\cite{DMPushforward}, \cite{DMKapustinLi}, \cite{PV}, \cite{Dyckerhoff},  \cite{Segal},    
\cite{MurfetResidues}, \cite{XuanThesis}, \cite{Platt}.
In particular, in this 
case, our definition of the Chern character of a matrix factorization is given by the
``Kapustin-Li formula'' \cite{KapLi}.

When the relative dimension $n$ of $X \to S$ is even,  
we obtain from $ch(\bE)$  a ``top Chern class''
$$
\ctop(\bE) \in H^0(X, \cJ_W\sst)
$$
where $\cJ_W$ is the ``Jacobian sheaf'' associated to  $W$. By definition, $\cJ_W$ is the coherent sheaf on $X$ given as the 
cokernel of the map $\Omega^{n-1}_{X/S} \otimes \cL^{-1} \xra{dW \smsh -}
\Omega^n_{X/S}$ define by multiplication by the element $dW$ of $\Gamma(X, \Omega^1_{X/S} \otimes_{\cO_X} \cL)$. 
The top Chern class contains less information than the Chern character in general, but in the affine case with $\cL$ trivial and under
certain other assumptions,
the two invariants coincide.

The main technical result of this paper concerns the vanishing of the top Chern class:

\begin{thm} \label{IntroThm1}
Let $k$ be a field of characteristic $0$ and  $Q$ a smooth $k$-algebra of even dimension. 
If $f_1, \dots, f_c \in Q$ is a
regular sequence of elements with  $c \geq 2$ such that the singular locus of $R = Q/(f_1, \dots, f_c)$ is zero-dimensional, 
then 
$$
\ctop(\bE) = 0 \in H^0(X, \cJ_W\sst)
$$ 
for all twisted matrix factorizations $\bE$ of $(\bP^{c-1}_Q, \cO(1), \sum_i f_i T_i)$. 
\end{thm}

As an application of this Theorem, we prove the vanishing of the invariants $\eta_c$ and $h_c$ for $c \geq 2$ in certain cases.
These invariants are defined for a pair of modules $M$ and $N$ over a complete
intersection 
$R = Q/(f_1, \dots, f_c)$ having the property that 
$\Ext^i_R(M,N)$ and $\Tor^i_R(M,N)$ are of finite length
for $i \gg 0$. For example, if the non-regular locus of $R$ is zero-dimensional, then every pair of finitely generated modules has this property.
In general, for such a pair of modules,
the even and odd sequences of lengths  of 
$\Ext^i_R(M,N)$ and $\Tor^i_R(M,N)$ 
are governed, eventually, by polynomials of degree at most $c-1$:
there exist polynomials 
$$
\begin{aligned}
p_{\ev}(M,N)(t) & = a_{c-1} t^{c-1} + \lots \\
p_{\odd}(M,N)(t) & = b_{c-1} t^{c-1} + \lots \\
\end{aligned}
$$
with integer coefficients such that 
$$
\len_R \Ext^{2i}_R(M,N) = p_{\ev}(M,N)(i)
\and
\len_R \Ext^{2i+1}_R(M,N) = p_{\odd}(M,N)(i)
$$
for $i \gg 0$,
and similarly for the $\Tor$ modules.
Following Celikbas and Dao \cite[3.3]{CDAsymptotic}, we define the
invariant $h_c(M,N)$ in terms of the leading coefficients of these polynomials: 
$$
h_c(M,N) := \frac{a_{c-1} - b_{c-1}}{c2^c}
$$
The invariant $\eta_c(M,N)$ is defined analogously, using $\Tor$ modules instead of $\Ext$ modules; see \cite{DaoAsymptotic}.

If $c = 1$ (i.e., $R$ is a hypersurface) the invariants $h_1$ and $\eta_1$ coincide (up to a factor of $\frac12$) 
with the {\em Herbrand difference}, defined originally by Buchweitz \cite{Buchweitz}, 
and {\em Hochster's $\theta$ invariant}, defined originally by Hochster \cite{Hochster}. 
In this case, $\Ext_R^i(M,N)$ and $\Tor^R_i(M,N)$ are finite length modules for $i \gg 0$ and 
the sequence of their
lengths are eventually two periodic, and $h_1$ and $\eta_1$ are determined by the formulas 
$$
2h_1^R(M,N) = h^R(M,n) = \len \Ext^R_{2i}(M,N) -  \len \Ext^R_{2i+1}(M,N), i \gg 0
$$
and
$$
2\eta_1^R(M,N) = \theta^R(M,n) = \len \Tor^R_{2i}(M,N) -  \len \Tor^R_{2i+1}(M,N), i \gg 0.
$$

We can now state the main application of Theorem \ref{IntroThm1}:

\begin{thm} \label{IntroThm}
Let $k$ be a field of characteristic $0$ and  $Q$ a smooth $k$-algebra.
If $f_1, \dots, f_c \in Q$ is a
regular sequence of elements with  $c \geq 2$ such that the singular locus of $R = Q/(f_1, \dots, f_c)$ is zero-dimensional, 
then  $h^R_c(M,N) = 0$ and $\eta^R_c(M,N) = 0$ for all finitely generated
  $R$-modules $M$ and $N$.
\end{thm}

\begin{ack} I am grateful to Jesse Burke, Olgur Celikbas, Hailong Dao, Daniel Murfet, and Roger Wiegand
for conversations about the topics of this paper.
\end{ack}

\section{Chern classes for generalized matrix factorizations} 

In this section, we develop ad hoc notions of Hochschild homology and Chern classes for 
twisted matrix factorizations. The connection with the singularity category for complete intersection rings  
will be explained more carefully later.

\begin{assumptions} \label{initialassume}
Throughout this section, we assume 
\begin{itemize}
\item $S$ is a Noetherian scheme. 
\item $p: X \to S$ is a smooth
morphism of relative dimension $n$. 

\item $\cL_S$ is  a locally free coherent sheaf
of rank one on $S$. Let $\cL_X := p^* \cL_S$, the pullback of $\cL_S$ along
$p$.
\item $W \in \Gamma(X, \cL_X)$ is global section of $\cL_X$.
\end{itemize}
\end{assumptions}

Recall that the smoothness assumption means that $p: X \to S$ is flat and of finite type
and  that $\Omega^1_{X/S}$ is a locally free coherent sheaf on $X$ of
rank $n$.

\begin{assumptions} \label{initialassume2}
When discussing ad hoc Hochschild homology and Chern characters (see below), we will
also be assuming:

\begin{itemize}
\item $p$ is affine and
\item $n!$ is invertible in $\Gamma(X, \cO_X)$.
\end{itemize}

\end{assumptions}

It is useful to visualize Assumptions \ref{initialassume} by a diagram. 
Let $q: L_S \to S$ denote the geometric line bundle whose sheaf of
sections is $\cL_S$; that is,  $L_S = \Spec(\bigoplus_{i \geq 0} \cL_S^{-i})$. 
Then we may interpret $W$ as a morphism $X \xra{W} L_S$ fitting into the
commutative triangle
$$
\xymatrix{
X \ar[d]_p \ar[dr]^W \\
S & L_S. \ar[l]^q\\
}
$$

The data $(p: X \to S, \cL, W)$ determines a closed subscheme $Y$ of $X$, defined by the vanishing of $W$. More formally, 
if $z: S \into L_S$ denotes the zero section of the line bundle $L_S$,  then $Y$ is defined by the pull-back square
$$
\xymatrix{
Y \ar[r] \ar[d] & X \ar[d]^W \\
S \ar[r]^z & L_S.
}
$$

We often suppress the subscripts and write $\cL$ to mean $\cL_S$
or $\cL_X$. For brevity, we write $\cF(n)$ to mean $\cF
\otimes_{\cO_X} \cL_X^{\otimes n}$ if
$\cF$ is a coherent sheaf or a complex of such on $X$. 
In our primary
application, $\cL$ will in fact by the standard very ample line bundle
$\cO(1)$ on projective space, and so this notation is
reasonable in this case.

The case of primary interest for us is described in the following example:

\begin{ex} \label{mainex}
Suppose $k$ is a field, $Q$ is a smooth $k$-algebra of dimension $n$, and 
$f_1, \dots, f_c$ is regular sequence  of elements of $Q$.
Let 
$$
\begin{aligned}
X & = \bP^{c-1}_Q = \Proj Q[T_1, \dots, T_c], \\
S & = \bP^{c-1}_k = \Proj k[T_1, \dots, T_c] 
\end{aligned}
$$
and define
$$
p: X \to S
$$
to be the map induced by $k \into Q$. Finally, set 
$\cL_S := \cO_S(1)$ and 
$$
W := \sum_i f_i T_i \in \Gamma(X, \cL_X).
$$ 

Then the triple $(p: X \to S, \cL, W)$ satisfies all the hypotheses
in Assumptions \ref{initialassume}, and moreover $p$ is affine. The subscheme cut out by $W$ is 
$$
Y = \Proj Q[T_1, \dots, T_c]/W
$$
in this case.
\end{ex}

\subsection{The Jacobian complex}
We construct a complex of locally free sheaves on $X$ that arises from
the data $(p: X \to S, \cL_S, W)$. Regarding $W$ as a map $W: X \to
L_S$ as above,
we obtain an induced map
$$
W^*  \Omega^1_{L_S/S} \to \Omega^1_{X/S}
$$
on cotangent bundles. There is a canonical isomorphism
$\Omega^1_{L_S/S} \cong q^*\cL^{-1}_S$ and hence an isomorphism
$\cL^{-1}_X \cong W^*  \Omega^1_{L_S/S}$. Using this,
we obtain a map $\cL^{-1}_X \to \Omega^1_{X/S}$, and upon tensoring
with $\cL$ we arrive at the map:
$$
dW: \cO_X \to \Omega^1_{X/S} \otimes_{\cO_X} \cL_X =: \Omega^1_{X/S}(1).
$$
Note that 
$$
\Lambda^\cdot_{X/S}(\Omega^1_{X/S} \otimes_{\cO_X} \cL) \cong
\bigoplus_q \Omega^q_{X/S}(q)
$$
so that we may regard the right-hand side as a sheaf of graded-commutative graded-$\cO_X$-algebras. Since $dW$ is a global section of this sheaf lying in degree one, 
we may form the {\em  Jacobian  complex} with differential given as repeated multiplication by $dW$:
\begin{equation} \label{E5}
\Omega_{dW}^{\cdot} := 
\left(
\cO_X \xra{dW} \Omega^1_{X/S}(1) 
 \xra{dW} \Omega^2_{X/S}(2) 
 \xra{dW} \cdots
 \xra{dW} 
 \Omega^n_{X/S}(n) \right).
\end{equation}
We index this complex so that $\Omega^j_{X/S}(j)$  lies in cohomological degree $j$.

\begin{ex} \label{ex:aff} 
Let us consider the case where $X$ and $S$ are affine and the line
bundle $\cL$ is the trivial one. That is, suppose $k$ is a Noetherian
ring, $A$ is a smooth $k$-algebra, and $f \in A$, and set $S =
\Spec(k)$, $X = \Spec(A)$,  $\cL = \cO_S$, and $W = f \in \Gamma(X,
\cO_X) = A$. (The map $p: X \to S$ is given by  the structural map $k \to A$.)

Then the Jacobian complex is formed by repeated multiplication by the degree one element 
$df \in \Omega^1_{A/k}$:
$$
\Omega^\cdot_{df} = \left(A \xra{df \smsh -} \Omega^1_{A/k} 
\xra{df \smsh -} \Omega^2_{A/k} 
\xra{df \smsh -} \cdots 
\xra{df \smsh -} 
\Omega^n_{A/k}\right).
$$

If we further specialize to the case $A = k[x_1, \dots, x_n]$, then we may
identify $\Omega^1_{A/k}$ with $A^n$ by using the basis $dx_1, \dots,
dx_n$. Then $\Omega^\cdot_{df}$ is the $A$-linear dual of the Koszul complex on the sequence
$\pd{f}{x_1}, \dots, \pd{f}{x_n}$ of partial derivatives of $f$.

If $k$ is a field and $f$ has only isolated singularities, then $\pd{f}{x_1}, \dots, \pd{f}{x_n}$ form an $A$-regular sequence. In this case, the Jacobian
complex has cohomology only in degree $n$ and
$$
H^n(\Omega^\cdot_{df}) \cong \frac{k[x_1, \dots, x_n]}{\langle  \pd{f}{x_1}, \dots, \pd{f}{x_n} \rangle} dx_1 \smsh \cdots \smsh dx_n.
$$
\end{ex}

\begin{ex} \label{mainex2}
With the notation of Example \ref{mainex}, the Jacobian complex is the
complex of coherent sheaves on $\bP^{c-1}_Q$ associated to the
following complex
of graded modules over the graded ring $Q[\vT] := Q[T_1, \dots, T_c]$:
$$
0 \to Q[\vT] \xra{dW} \Omega^1_{Q/k}[\vT](1) \xra{dW} \cdots
\xra{dW} \Omega^n_{Q/k}[\vT](n) \to 0
$$
Here, $dW = df_1 T_1 + \cdots + df_c T_1$, a degree one element in the
graded module $\Omega^1_{Q/k}[\vT]$.

Further specializing to $Q = k[x_1, \dots, x_n]$, we have $df_i = \sum_j
\pd{f}{x_j} dx_j$ and 
$$
dW = 
\begin{bmatrix}
\pd{f_1}{x_1} & \cdots & \pd{f_c}{x_1} \\
\pd{f_1}{x_2} & \cdots & \pd{f_c}{x_2} \\
\vdots & & \vdots \\
\pd{f_1}{x_n} & \cdots & \pd{f_c}{x_n} \\
\end{bmatrix}
\cdot
\begin{bmatrix}
T_1 \\
T_2 \\
\vdots \\ 
T_c
\end{bmatrix}.
$$

It is convenient to think of $\Omega^\cdot_{dW}$ as representing a ``family''
of complexes of the type occurring in Example \ref{ex:aff} indexed by
the $Q$ points of $\bP^{c-1}_Q$. That is,
given a tuple
$\va = (a_1, \dots, a_c)$ of elements of $Q$ that generated the unit ideal,
we have an associated $Q$-point $i_{\va}: \Spec(Q)  \into
\bP^{c-1}_Q$ of $\bP^{c-1}_Q$. Pulling back along $i_{\va}$ (and
identifying $i_\va^* \cO(1)$ with the trivial bundle in the canonical
way) gives the
Jacobian complex of Example \ref{ex:aff} for $(\Spec(Q) \to \Spec(k), \cO, \sum_i a_i f_i)$.
\end{ex}

We are especially interested in the cokernel (up to a twist) of the last map in
the Jacobian complex \eqref{E5}, and thus give it its own name:

\begin{defn} 
The {\em Jacobi sheaf} associated to 
$(p: X \to S, \cL, W)$, written 
$\cJ(X/S, \cL, W)$ or just $\cJ_W$ for short, is the coherent sheaf on $X$
  defined as
$$
\cJ_W = \cJ(X/S, \cL, W) := \coker\left(\Omega^{n-1}_{X/S}(-1)
  \xra{dW \smsh -} \Omega^{n}_{X/S}\right).
$$
In other words, $\cJ_W := \cH^n(\Omega_{dW}^{\cdot})(-n)$, where $\cH^n$ denote the cohomology of a complex in the abelian category of coherent sheaves on $X$. 
\end{defn}

\begin{ex} If $S = \Spec(k)$, $k$ is a field, $X = \A^n_k = \Spec(k[x_1, \dots, x_n])$, $\cL = \cO_X$ and $f$ a polynomial, then
 $$
\cJ_f = \frac{k[x_1, \dots, x_n]}{\left(\pd{f}{x_1}, \dots,
\pd{f}{x_n}\right)} dx_1 \smsh \cdots \smsh dx_n.
$$
\end{ex}

\begin{ex} \label{mainex3}
With the notation of \ref{mainex}, $\cJ_W$ is the coherent sheaf
associated to the graded $Q[\vT]$-module
$$
\coker\left(\Omega^{n-1}_{Q/k}[\vT](-1) \xra{dW} 
\Omega^{n}_{Q/k}[\vT]\right).
$$
If we further specialize to the case $Q = k[x_1, \dots, x_n] = k[\vx]$, then using $dx_1, \dots, dx_n$ as a basis
of $\Omega^1_{Q/k}$ gives us that $\cJ_W$ is isomorphic to the
coherent sheaf associated to  
$$
\coker\left(
k[\vx, \vT](-1)^{\oplus n}
\xra{(T_1, \dots, T_c) \left(\pd{f_i}{x_i}\right)^t}
k[\vx, \vT]\right).
$$
\end{ex}

\subsection{Matrix factorizations}
We recall the theory of ``twisted matrix factorizations'' from \cite{PVStack} and \cite{BW1}.

Suppose $X$ is any Noetherian scheme, $\cL$ is a line bundle on $X$ and $W \in
\Gamma(X, \cL)$ any global section of it. (The base scheme $S$ is not
needed for this subsection, and we write $\cL$ for $\cL_X$.)
A {\em twisted matrix factorization for $(X, \cL, W)$} consists of the data $\bE = (\cE_0, \cE_1, d_0, d_1)$ where  $\cE_0$ and
$\cE_1$ 
are locally free coherent sheaves on $X$ and 
$d_0: \cE_0 \to \cE_1(1)$  and
$d_1: \cE_{1} \to \cE_0$ 
are morphisms such that $d_0 \circ d_{1}$ and $d_{1}(1) \circ d_0$
are each given by multiplication by $W$. (Recall that $\cE_1(1)$ denotes $\cE_1 \otimes_{\cO_X} \cL$.) 
We visualize $\bE$ as a diagram of the form
$$
\bE = \left(\cE_0 \xra{d_0} \cE_1 \xra{d_1} \cE_0(1) \right)
$$
or as the ``twisted periodic'' sequence
$$
\cdots \xra{d_0(-1)} \cE_1(-1) \xra{d_1(-1)}
\cE_0 \xra{d_0} \cE_1 \xra{d_1}\cE_0(1) \xra{d_0(1)} \cE_1(1)
\xra{d_1(1)} \cdots.
$$
Note that the composition of any two adjacent arrows in this sequence
is multiplication by $W$ (and hence this is not a complex unless $W = 0$).

A {\em strict morphism} of matrix factorizations for $(X, \cL, W)$,
say from $\bE$ to $\bE' = (\cE_0', \cE_1', d'_0, d'_1)$ consists of a pair of morphisms 
$\cE_0 \to \cE_0'$ and $\cE_1 \to \cE_1'$ such that the evident
pair of squares both commute.
Matrix factorizations and strict morphisms form an exact category, which we write as 
$mf(X, \cL, W)$, for which  a sequence of morphisms is declared exact if it
is so in each degree.

If $V \in \Gamma(X, \cL)$ is another global section,
there is a tensor product pairing
$$
- \otimes_{mf} - : mf(X, \cL, W) \times mf(X, \cL, V) \to mf(X, \cL, W + V).
$$
In the special case $\cL = \cO_X$ and $V = W = 0$, this tensor product is the usual tensor product of $\Z/2$-graded complexes of $\cO_X$-modules, and the
general case is the natural ``twisted'' generalization of this. We refer the reader to \cite[\S7]{BW2} for a more precise definition.

There is also a sort of internal $\Hom$ construction. Given $W, V \in
\Gamma(X, \cL)$ and objects $\bE \in mf(X, \cL, W)$ and $\bF \in mf(X,
\cL, V)$, there is an object  $\HomMF(\bE, \bF) \in mf(X, \cL, V - W)$. Again, if $\cL = \cO_X$ and $W=V=0$, this is the usual internal $\Hom$-complex of a pair of
$\Z/2$-graded complexes of $\cO_X$-modules, and we refer the reader to \cite[2.3]{BW2} for the general definition.
If $V = W$, then $\HomMF(\bE, \bF)$ belongs to $mf(X, \cL,
0)$ and hence is actually a complex. In general, we refer to an object of $mf(X, \cL, 0)$ as a  {\em twisted two-periodic complex}, since it is a complex of
the form
$$
 \cdots \xra{d_0(-1)} \cE_1 \xra{d_1} \cE_0 \xra{d_0} \cE_1(1) \xra{d_1(1)} \cE_0(1) \xra{d_0(1)} \cdots.
$$
In other words, such an object is equivalent to the data of a (homologically
indexed) complex $\cE_\cdot$
of locally free coherent sheaves and a specified isomorphism
$\cE_\cdot(1) \cong \cE_\cdot[2]$. (We use the convention that
$\cE_\cdot[2]_i = \cE_{i-2}$.)

There is also a notion of duality. Given $\bE \in mf(X, \cL, W)$, we
define $\bE^* \in mf(X, \cL, -W)$ as $\HomMF(\bE, \cO_X)$ where here $\cO_X$ denotes the
object of $mf(X, \cL, 0)$ given by the twisted two-periodic complex
$$
\cdots \to 0 \to \cO_X \to 0 \to \cL \to 0 \to \cL^{\otimes 2} \to
\cdots.
$$
(In the notation introduced below, $\bE^* = \HomMF(\bE, \fold \cO_X)$.)  
See \cite[7.5]{BW2} for a more explicit formula.

These constructions are related by the
natural isomorphisms (cf. \cite[\S 7]{BW2})
$$
\HomMF(\bE, \bF) \cong \bE^* \otimes_{mf} \bF
$$
and
$$
\HomMF(\bE \otimes_{mf} \bG, \bF) \cong \HomMF(\bE, \HomMF(\bG, \bF)).
$$

The construction of $\HomMF(\bE, \bF)$ is such that we have a natural identification
$$
\Hom_{strict}(\bE, \bF) = Z_0\Gamma(X, \HomMF(\bE, \bF)).
$$
Here, $\Gamma(X, \HomMF(\bE, \bF))$ is the complex of abelian groups obtained
by regarding $\HomMF(\bE, \bF)$ as an unbounded complex of coherent sheaves on $X$ and applying
the global sections functor $\Gamma(X, -)$ degree-wise, obtaining a complex of $\Gamma(X, \cO_X)$-modules, and $Z_0$
denotes the cycles lying in degree $0$ in this complex.

Let $\cP(X)$ denote the category of complexes of locally free
coherent sheaves on $X$ and let $\cP^b(X)$ denote the full subcategory
of bounded complexes.
There is a ``folding''   functor,
$$
\Fold: \cP^b(X) \to  mf(X, \cL, 0),
$$
defined by taking direct sums of the even and odd terms in a complex. In detail, given a
complex $(P, d_P) \in \cP^b(X)$, let  
$$
\Fold(P)_0 = \bigoplus_i P_{2i}(i)\and
\Fold(P)_1 = \bigoplus_i P_{2i+1}(i).
$$
The required map $\Fold(P)_1 \to \Fold(P)_0$ maps the summand $P_{2i+1}(i)$
to the summand $P_{2i}(i)$ via $d_P(i)$ and 
$\Fold(P)_0 \to \Fold(P)_1(1)$ maps $P_{2i}(i)$ to $P_{2i-1}(i)$ also
via $d_P(i)$. 

There is also an ``unfolding'' functor
$$
\unfold : mf(X, \cL, 0) \to \cP(X)
$$
which forgets the two periodicity. In detail,
$$
\unfold(\cE_1 \to \cE_0 \to \cE_1(1))_{2i} = \cE_0(-i) \and
\unfold(\cE_1 \to \cE_0 \to \cE_1(1))_{2i+1} = \cE_1(-i)
$$
with the evident maps.

These functors are related by the following adjointness properties:
Given $\bE \in mf(X, \cL, 0)$ and $\cP \in \cP^b(X)$, there are
natural isomorphisms
$$
\Hom_{mf(X, \cL,0)}(\Fold \cP, \bE) \cong
\Hom_{\cP(X)}(\cP, \unfold \bE)
$$
and
$$
\Hom_{\cP(X)}(\unfold \bE, \cP)
\cong
\Hom_{mf(X, \cL,0)}(\bE, \Fold \cP).
$$
(These functors do not technically form an adjoint pair, since $\unfold$ takes
values in $\cP(X)$, not $\cP^b(X)$.)

\begin{ex} If $\cE$ is a locally free coherent sheaf on $X$, we write
  $\cE[0]$ for the object in $\cP(X)$ with $\cE$ in 
  degree $0$ and $0$'s elsewhere. Then 
$$
\begin{aligned}
\Fold(\cE[0]) & = \left(0 \to \cE \to 0 \right) \\
& = \left( \cdots \to \cE(-1) \to 0 \to \cE \to 0 \to
  \cE(1) \to \cdots \right)
\end{aligned}
$$
indexed so that $\cE$ lies in degree $0$. Letting $\cE[j] = \cE[0][j]$, we have 
$$
\begin{aligned}
\Fold(\cE[2i]) & = \left(0 \to \cE(i) \to 0\right) \\
& =
\left( \cdots \to \cE(i-1) \to 0 \to \cE(i) \to 0 \to
  \cE(i+1) \to \cdots \right) \\
\end{aligned}
$$
with $\cE(i)$ in degree $0$, and
$$
\begin{aligned}
\Fold(\cE[2i+1]) & = \left(\cE(i) \to 0 \to \cE(i+1)\right) \\
& =
\left( \cdots \to \cE(i-1) \to 0 \to \cE(i) \to 0 \to
  \cE(i+1) \to \cdots \right) \\
\end{aligned}
$$
with $\cE(i)$ in homological degree $1$.
\end{ex}

If $P \in \cP^b(X)$ is a bounded complex and $\bE \in mf(X, \cL, W)$ is a twisted matrix factorization, we define their tensor
product to be 
$$
P \otimes \bE := \Fold(P) \otimes_{mf} \bE \in mf(X, \cL, W).
$$

\begin{ex}
We will use this construction, in particular, for the complex $P =
\Omega^1_{X/S}(1)[-1]$ consisting of $\Omega^1_{X/S}(1)$ concentrated in
homological degree $-1$. Since $\Omega^1$ lies in odd degree, the sign conventions give us
$$
\Omega^1_{X/S}(1)[-1] \otimes \bE = \left(
\Omega^1_{X/S} \otimes \cE_0
\xra{- \id \otimes d_0}
\Omega^1_{X/S} \otimes \cE_1(1)
\xra{-\id \otimes d_1(1)}
\Omega^1_{X/S} \otimes \cE_0(1)\right).
$$

We will also apply this construction to the complex
$P := \left(\cO_X \xra{dW \smsh -} \Omega^1_{X/S}(1)\right)$ indexed so that $\cO_X$
lies in degree $0$ to form 
the matrix factorization
$$
(\cO_X \xra{dW \smsh -} \Omega^1_{X/S}(1)) \otimes_{\cO_X} \bE
$$
which is given explicitly as
$$
\voplus{\cE_1}{\Omega^1_{X/S} \otimes \cE_0}
\xra{\scriptstyle \begin{bmatrix}
d_1 & 0 \\
dW & - \id \otimes d_0 \\
\end{bmatrix}}
\voplus{\cE_0 }{\Omega^1_{X/S} \otimes \cE_1(1)}
\xra{\begin{bmatrix}
d_0 & 0 \\
 dW & - \id \otimes d_1 \\
\end{bmatrix}}
\voplus{\cE_1(1)}{\Omega^1_{X/S} \otimes \cE_0(1)}.
$$
\end{ex}

Building on the identity
$$
\Hom_{strict}(\bE, \bF) = Z_0 \Gamma(X, \HomMF(\bE, \bF))
$$
we obtain a natural notion of homotopy: 
Two strict morphisms are {\em homotopic} if their difference
lies in the image of the boundary map in the complex $\Gamma(X, \HomMF(\bE, \bF))$.
That is, the group of equivalence classes of strict morphisms up to homotopy from $\bE$
to $\bF$ is $H_0 \Gamma(X, \HomMF(\bF, \bE))$.
We define $[mf(X, \cL, W)]_\naive$, the ``naive homotopy category'', to be the category with the same objects
as 
$mf(X, \cL, W)$ and with morphisms given by strict maps modulo
homotopy; that is,
$$
\Hom_{[mf]_\naive}(\bE, \bF) := H_0 \Gamma(X, \HomMF(\bF, \bE)).
$$
Just as the homotopy category of chain complexes is a triangulated category, so too is the category $[mf(X, \cL, W)]_\naive$;
see \cite[1.3]{PVStack} or \cite[2.5]{BW1} for details.

The reason for the pejorative ``naive'' in this definition is that
this notion of homotopy equivalence does not globalize well: two
morphisms can be locally homotopic without being globally
so. Equivalently, an objects can be locally contractible with being
globally so. To rectify this, 
we define the {\em homotopy category of matrix factorizations} for $(X,
\cL, W)$, written $hmf(X, \cL, W)$, to be the Verdier quotient of $[mf(X,
\cL, W)]_\naive$ by the thick subcategory consisting of objects that
are locally contractible. (The notion appears to be originally due to Orlov, but see also \cite[3.13]{PVStack} where this category is referred to as the
``derived category of matrix factorizations''.)

We list the properties of $hmf(X, \cL, W)$ needed in the rest of this
paper: 

\begin{enumerate}
\item There is a functor $mf(X, \cL, W) \to hmf(X, \cL, W)$, and it sends
  objects that are locally contractible to the trivial object.

\item If
  $\alpha: \bE \to \bF$ is strict morphism and it is locally
  null-homotopic, then $\alpha$ is sent to the $0$ map in $hmf(X, \cL,
  W)$.

\item When $W = 0$, $hmf(X, \cL, 0)$ is the Verdier quotient of
  $[mf(X, \cL, 0]_\naive$ obtained by inverting quasi-isomorphisms of twisted
  two-periodic complexes. 

\item At least in certain cases, the hom sets of $hmf(X, \cL, W)$
  admit a more explicit description. In general, 
for objects $\bE, \bF \in mf(X, \cL, W)$ there is a natural
map 
$$
\Hom_{hmf}(\bE, \bF) \to 
\bH^0(X, \unfold \HomMF(\bE, \bF))
$$
where $\bH^0(X, -)$ denotes sheaf hyper-cohomology \cite[3.5]{BW1}. This
map is an isomorphism when 
\begin{enumerate}
\item $W = 0$ or  
\item $X$ is projective over an affine base scheme and $\cL =
  \cO_X(1)$ is the standard line bundle; see \cite[4.2]{BW1}.
\end{enumerate}

\end{enumerate}

\subsection{Connections for matrix factorizations and the Atiyah class}

See Appendix \ref{relcon} for recollections on the notion of a
connection for a locally free
coherent sheaf on a scheme.  Much of the material in this subsection
represents a generalization to the non-affine case of constructions
found in \cite{XuanThesis}, which were in turn inspired by 
constructions in \cite{DMPushforward}.

\begin{defn} With $(p:X \to S, \cL, W)$ as in Assumptions \ref{initialassume}, given 
$\bE \in
    mf(X, \cL, W)$,
a {\em connection on $\bE$
    relative to $p$}, written as $\bnabla_\bE$ or just $\bnabla$,  is a pair of connections
relative to $p$
$$
\nabla_1: \cE_1 \to \Omega^1_{X/S} \otimes_{\cO_X} \cE_1
$$
and
$$
\nabla_0: \cE_0 \to \Omega^1_{X/S} \otimes_{\cO_X} \cE_0.
$$
There is no condition relating $\bnabla_\bE$ and
$d_1, d_0$. 
\end{defn}

Since each $\nabla_i$ is $p^*\cO_S$-linear and $\cL$ is pulled back from $S$,
we have induced connections
$$
\nabla_1(j): \cE_1(j) \to \Omega^1_{X/S} \otimes_{\cO_X} \cE_1(j)
$$
and
$$
\nabla_0(j): \cE_0(j) \to \Omega^1_{X/S} \otimes_{\cO_X} \cE_0(j)
$$
for all $j$.

\begin{ex} \label{Ex816}
Suppose $S = \Spec k$ and $X = \Spec Q$ are affine
where  $Q$ is a smooth $k$-algebra of dimension $n$, $\cL = \cO_X$ so that $W \in Q$, and the
  components of $\bE$ are free $Q$-modules. Upon choosing bases, we
  may represent $\bE$ as $\left( Q^r \darrow{A}{B} Q^r\right)$ where
  $A, B$ are $r \times r$ matrices with entries in $Q$ such that $AB =
  BA = WI_r$. The choice of basis leads to an associated ``trivial'' connection 
$$
\nabla = d:
  Q^r \to \Omega^1_{Q/k} \otimes_Q Q^r = \left(
    \Omega_{Q/k}^1\right)^{\oplus r}
$$
given by applying exterior differentiation to the components of a vector.
\end{ex}

As noted in Appendix \ref{relcon}, if $p$ is affine, then every vector bundle
on $X$ admits a connection relative to $p$ and so every
matrix factorization admits a connection in this case.

\begin{defn} Given a connection $\nabla$ relative to $p$ for a twisted matrix factorization
  $\bE \in mf(X, \cL, W)$, the associated {\em Atiyah class} is defined to be  the map
$$
\At_{\bE, \nabla}: \bE \to \Omega^1_{X/S}(1)[-1] \otimes \bE
$$
given as the ``commutator'' $[\nabla, \delta_\bE]$. In detail, 
it is given by the pair of maps
$$
\nabla_1(1) \circ \delta  - (\id \otimes \delta) \circ \nabla_0:
\cE_0 \to \Omega^1(1) \otimes \cE_1
$$
and
$$
\nabla_0 \circ \delta  - (\id \otimes \delta) \circ \nabla_1:
\cE_1 \to \Omega^1 \otimes \cE_0.
$$
\end{defn}

$\At_{\bE, \nabla}$ is not a morphism of matrix factorizations in
general, but by Lemma \ref{lem:comm} we have that $\At_{\bE, \nabla}$ is $\cO_X$-linear.

\begin{ex} \label{Ex816b}
Keep the notations and assumptions of Example
  \ref{Ex816}, so that
  $$
\bE = \left(Q^r \darrow{A}{B} Q^r\right).
$$ 
Then
$$
\Omega^1(1)[-1] \otimes \bE =
\left( \left(\Omega^1_{Q/k}\right)^{\oplus r} \darrow{-B}{-A} 
\left(\Omega^1_{Q/k}\right)^{\oplus r}\right)
$$
and
the map $\At: \bE \to \Omega^1(1)[-1] \otimes \bE$ obtained by choosing the
  trivial connections is represented by the following diagram:
$$
\xymatrix{
Q^r \ar[r]^{A} \ar[d]^{dA} & Q^r \ar[d]^{dB} \ar[r]^{B}&  Q^r \ar[d]^{dA}  \\
\left(\Omega^1_{Q/k}\right)^{\oplus r}\ar[r]^{-B} & 
\left(\Omega^1_{Q/k}\right)^{\oplus r}
\ar[r]^{-A} &
\left(\Omega^1_{Q/k}\right)^{\oplus r}
\\
}
$$
Here $dA$ and $dB$ denote the $r \times r$ matrices with entires in $\Omega^1_{Q/k}$ obtained by applying $d$ entry-wise to $A$ and $B$.

This is not a map of matrix factorizations since the
squares do not commute. Indeed, the differences of the compositions
around these squares are $(dB)A + B(dA)$  and $(dA)B + A(dB)$. 
Since $AB = BA = W$, these expression both equal $dW$, a fact which is
relevant for the next construction.
\end{ex}

\begin{defn} Given a connection $\nabla$ for
 a matrix factorization $\bE$ , define
$$
\Psi_{\bE, \bnabla}: \bE \to \left(\cO_X \xra{dW \smsh -}
  \Omega^1_{X/S}(1)\right)  \otimes_{\cO_X} \bE
$$
as $\id_\bE + \At_\nabla$. (Here, $\cO_X \xra{dW \smsh -}
\Omega^1_{X/S}(1)$ is the complex with $\cO_X$ in degree $0$.)
In other words, $\Psi_{\bE, \nabla}$ is the morphism
whose composition with the canonical projection 
$$
\left(\cO_X \xra{dW \smsh -} \Omega^1_{X/S}(1)\right) \otimes_{\cO_X} \bE 
\onto
\bE
$$
is the identity and whose composition with the canonical projection
$$
\left(\cO_X \xra{dW \smsh -} \Omega^1_{X/S}(1)\right) \otimes_{\cO_X} \bE 
\onto
\Omega^1(1)[-1] \otimes_{\cO_X}\bE
$$
is $\At_{\bE,\nabla}$.
\end{defn}

\begin{lem} \label{lem:ind}
The map $\Psi_{\bE, \bnabla}$ is a (strict) morphism of matrix
  factorizations, and it is independent in the
naive  homotopy category $[mf(X, \cL, W)]_\naive$ (and hence the homotopy category too) of the choice of connection
  $\bnabla$.
\end{lem}

\begin{proof} The proofs found in \cite{XuanThesis}, which deal with the
  case where $X$ is affine and $\cL$ is trivial, apply nearly
  verbatim. The homotopy relating $\Psi_{\bE, \bnabla}$ and 
$\Psi_{\bE, \bnabla'}$ for two different connections $\bnabla$ and
$\bnabla'$ on $\bE$ is given by the map
$$
\bnabla - \bnabla': \bE \to
\Omega^1_{X/S} \otimes \bE,
$$ 
which is $\cO_X$-linear by Lemma \ref{lem:comm}.
\end{proof}

\begin{ex} \label{Ex816c}
Continuing with Examples \ref{Ex816} and \ref{Ex816b},  the map
$\Psi$ in this case is represented by the diagram 
$$
\xymatrix{
Q^r \ar[rr]^A \ar[d]^{\tiny \vectwo{1}{dA}} && Q^r \ar[rr]^B
  \ar[d]^{\tiny \vectwo{1}{dB}} && Q^r \ar[d]^{\tiny \vectwo{1}{dA}}  \\
\voplus{Q^r}{\left(\Omega^1_{Q/k}\right)^r} 
\ar[rr]_{\tiny \begin{bmatrix} A & 0 \\ dW & -B \end{bmatrix}} &&
\voplus{Q^r}{\left(\Omega^1_{Q/k}\right)^r} 
\ar[rr]_{\tiny \begin{bmatrix} B & 0 \\ dW & -A \end{bmatrix}} &&
\voplus{Q^r}{\left(\Omega^1_{Q/k}\right)^r.} 
}
$$
This diagram commutes, confirming that $\Psi$ is indeed a morphism of
matrix factorizations.
\end{ex}

\subsection{Ad hoc Hochschild homology}

For an integer $j \geq 0$, we define the twisted two-periodic complex
$$
\begin{aligned}
& \Omega^{(j)}_{dW} := \\
&
\fold 
\left(\cO_X \sxra{j dW} 
\Omega^1_{X/S}(1)
\sxra{(j-1) dW}  \Omega^2_{X/S}(2)
\sxra{(j-2) dW}  \cdots
\sxra{2 dW}  
 \Omega^{j-1}_{X/S}(j-1)
\sxra{dW}  
\Omega^j_{X/S}(j))\right). \\
\end{aligned}
$$

Explicitly, $\Omega^{(j)}_{dW}$ is the twisted periodic
complex 
$$
\cdots \lra
\begin{bmatrix}
\Omega^1_{X/S} \\
\oplus \\
\Omega^3_{X/S}(1) \\
\oplus \\
\Omega^5_{X/S}(2) \\
 \oplus \\
\vdots
\end{bmatrix}
\lra
\begin{bmatrix}
\cO_X \\
\oplus \\
\Omega^2_{X/S}(1) \\
\oplus \\
\Omega^4_{X/S}(2) \\
 \oplus \\
\vdots
\end{bmatrix}
\lra
\begin{bmatrix}
\Omega^1_{X/S}(1) \\
\oplus \\
\Omega^3_{X/S}(2) \\
\oplus \\
\Omega^5_{X/S}(3) \\
 \oplus \\
\vdots
\end{bmatrix}
\lra \cdots
$$
with 
$\bigoplus_{i = 0}^{i = \floor{\frac{j}{2}}} \Omega^{2i}_{X/S}(i)$ in degree $0$
and
$\bigoplus_{i = 0}^{i = \floor{\frac{j-1}{2}}} \Omega^{2i+1}_{X/S}(i)$
in degree $1$.

The coefficients appearing in the maps of $\Omega^{(j)}_{dW}$ are
necessarily to make the following statement hold true. (We omit its straight-forward proof.)

\begin{lem} The pairings 
$$
\Omega^i_{X/S}(i) \otimes_{\cO_X}
  \Omega^l_{X/S}(l) \to 
  \Omega^{i+l}_{X/S}(i+l) 
$$
defined by exterior product induce a morphism
$$
\Omega_{dW}^{(j)} \otimes_{mf} \Omega_{dW}^{(l)} \to
\Omega_{dW}^{(j+l)}.
$$
in $mf(X, \cL, 0)$.
\end{lem}

If $j!$ is invertible in $\Gamma(X, \cO_X)$ (for example, if $X$ is a
scheme over a field $k$ with $\chr(k) = 0$ or $\chr(k) > j$), then
there is an isomorphism
$$
\Omega^{(j)}_{dW}
\cong 
\fold 
\left(\cO_X \sxra{dW} 
\Omega^1_{X/S}(1)
\sxra{dW}  \cdots
\sxra{dW}  
\Omega^j_{X/S}(j))\right)
$$
given by collection of isomorphisms $\Omega_{X/S}^i(i)
\xra{\frac{1}{n(n-1) \cdots (n-i)}} \Omega_{X/S}^i(i)$, $0 \leq i \leq
j$. In particular, if $n!$ is invertible in $\Gamma(X, \cO_X)$, then there
is an isomorphism
$$
\Omega^{(n)}_{dW} \cong \fold \Omega_{dW}^{\cdot}.
$$
(Recall that $n$ is the relative dimension of $X$ of $S$, and hence is
the rank of $\Omega^1_{X/S}$. It follows that
$\Omega^m_{X/S} = 0$ for $m > n$.)

\begin{defn} Let $p: X \to S, \cL, W$ be as in Assumptions \ref{initialassume} and
  assume also that $n!$ is invertible in $\Gamma(X, \cO_X)$. We define
  the (degree $0$) {\em ad hoc Hochschild homology of $mf(X, \cL, W)$ relative to
    $S$} 
to be 
$$
\HHn_0(X/S, \cL, W):= \bH^0(X, \unfold \fold \Omega^\cdot_{dW}).
$$
\end{defn}

As we will see in the next section, $\HHn_0(X/S, \cL, W)$ is the
target of what we term the ``ad hoc Chern character'' of a twisted matrix
factorization belonging to  $mf(X, \cL, W)$. 

We use the adjective ``ad hoc'' for two reasons. The first is that 
we offer here no justification that this definition
deserves to be called ``Hochschild homology''.
See, however, \cite{PV, Preygel} for the affine case and \cite{Platt}
for the non-affine case with $\cL = \cO_X$.

The second reason is
that whereas the support of every object of $mf(X, \cL, W)$ is contained in the
singular locus of the subscheme $Y$ cut out by $W$ (as will be
justified in the next subsection), the twisted two-periodic complex
$\fold \Omega_{dW}^{\cdot}$ is not so supported in general. A more
natural definition of Hochshild homology would thus be given by 
$\bH^0_Z (X, \unfold \fold \Omega^\cdot_{dW})$, where $Z$ is the
singular locus of $Y$ and $\bH^0_Z$ refers to hypercomology with
supports (i.e., local cohomology). Since this more sensible definition
of Hochschild homology is not necessary for our purposes and only adds
complications to what we do, we will stick with using $\HHn_0$.

Recall that the Jacobi sheaf is
defined as
$$
\cJ_W = \cJ(X/S, \cL, W) = \coker\left(\Omega^{n-1}_{X/S}\left(-1\right) \xra{dW} 
\Omega^{n}_{X/S}\right)
$$
so that there is a canonical map
$$
\Omega^\cdot_{dW} \to
\cH^n(\Omega^\cdot_{dW})[-n] = 
\cJ_W(n)[-n].
$$
From it we obtain the map
$$
\fold (\Omega^\cdot_{dW}) \to
\fold (\cJ_W(n)[-n]) 
\cong \fold(\cJ_W\sst).
$$
Applying $\unfold$ and using the canonical map
$\unfold \fold(\cJ_W\sst) \to \cJ_W\sst$
results in a map 
$\unfold \fold \Omega^\cdot_{dW} \to \cJ_W\sst$. Finally, applying
$\bH^0(X, -)$ yields the map
\begin{equation} \label{E38}
\HHn_0(X/S, \cL, W) \to H^0(X, \cJ_W\sst).
\end{equation}
This map will play an important role in the rest of this paper. Let us
observe that in certain situations, it is an isomorphism:

\begin{prop} \label{P131}
Assume $S= \Spec k$ for a field $k$, $X = \Spec(Q)$ for a
  smooth $k$-algebra $Q$, and hence $\cL_X = \cO_X$ and $W \in Q$. If
the morphism of smooth varieties   $W: X \to \bA^1$ has only isolated critical points and $n$ is even,
  then the canonical map
$$
\HHn_0(X/S, W) \to H^0(X, \cJ_W\sst) = \frac{\Omega^n_{Q/k}}{dW \smsh
  \Omega^{n-1}_{Q/k}}
$$
is an isomorphism.
\end{prop}

\begin{proof} These conditions ensure that the complex of $Q$-modules
  $\Omega^\cdot_{dW}$ is exact except on the far right where it has
  homology $\cJ(n)$.  The result follows since $X$ is affine.
\end{proof}

\begin{ex} Assume $S = \Spec(k)$ for a field $k$,  $X = \A^n_k = \Spec
  k[x_1, \dots, x_n]$, and the morphism $f: \A^n_k \to \A^1_1$ associated to a given polynomial $f \in k[x_1, \dots, x_n]$ has only isolated
  critical points. Then 
$$
\HHn_0(X/S, f) \cong 
H^0(X, \cJ_f\sst) =
\frac{\Omega^n_{Q/k}}{df \smsh
  \Omega^{n-1}_{Q/k}}  =\frac{k[x_1, \dots, x_n]}{(\pd{f}{x_1},
\dots, \pd{f}{x_n})} dx_1 \smsh \cdots \smsh dx_n.
$$
\end{ex}

\subsection{Supports} 

We fix some notation.
For a Noetherian scheme $Y$,  the {\em non-regular locus} of
$Y$ is
$$
\Nonreg(Y) := \{y \in Y \, | \, \text{ the local ring $\cO_{Y,y}$ is
  not a regular local ring} \}.
$$
Under mild additional hypotheses (e.g., $Y$ is excellent), $\Nonreg(Y)$ is a closed
subset of $Y$.

Assume $g: Y \to S$ is a morphism of finite type. Recall that $g$ is smooth near $y \in
Y$ if and only if it is flat of relative dimension $n$ near $y$ and the stalk of
$\Omega^1_{Y/S}$ at $y$ is a free $\cO_{Y,y}$-module of rank $n$.
We define {\em singular
  locus} of $g: Y \to S$ to be the subset  
$$
\Sing(g) = \Sing(Y/S) = \{y \in Y \, | \, \text{ $g$ is not smooth
  near $y$} \}.
$$
At least for a flat morphism $g: Y \to S$ of finite type, the singular
locus of $g$ is a closed subset of $Y$ by the Jacobi criterion. 
When $S = \Spec(k)$ for a field $k$ and there is no danger of
confusion, we  write $\Sing(Y)$ instead of $\Sing(Y/\Spec(k))$.
If $Y$ is finite type over a field $k$,
then $\Nonreg(Y) \subseteq \Sing(Y) = \Sing(Y/\Spec(k))$, and equality holds if
$k$ is perfect.

\begin{prop} \label{prop38}
Assume $X$ is a Noetherian scheme, $\cL$ is a line
  bundle on $X$, and $W$ is a regular global section of $\cL$. Let 
 $Y \subseteq X$ be the closed subscheme cut out by $W$.
Then for
  every $\bE, \bF \in mf(X, \cL, W)$, the twisted two-periodic complex $\HomMF(\bE, \bF)$
  is supported 
$\Nonreg(Y)$.
\end{prop}

\begin{proof} For $x \in X$, we have
$$
\HomMF(\bE, \bF)_x
\cong 
\HomMF(\bE_x, \bF_x)
$$
where the hom complex on the right is for the category
$mf(\Spec(\cO_{X,x}), \cL_x, W_x)$. By choosing a trivialization
$\cO_X \cong \cL_x$, we may identify this category with $mf(Q,f)$ for
a regular local ring $Q$ and non-zero-divisor $f$. The assertion thus
becomes that if $f$ is non-zero-divisor in regular local ring $Q$ and
$Q/f$ is also regular, then the complex $\HomMF(\bE, \bF)$
is acyclic for all objects $\bE, \bF \in mf(Q,f)$. This holds since
the cohomology modules of the complex
$\HomMF(\bE, \bF)$ are $\sExt^*_{Q/f}(\coker(\bE), \coker(\bF))$; see
Theorem \ref{thm1030} below.
\end{proof}

\begin{ex} \label{ex:1017}
Suppose $k$ is a field, $Q$ is a smooth $k$-algebra, $\fm$
  is a maximal ideal of $Q$, and $f_1, \dots, f_c \in \fm$ form a
  regular sequence such that the non-regular locus of the ring  $R := Q/(f_1, \dots, f_c)$ 
is $\{\fm\}$.
Let $X = \bP^{c-1}_Q = \Proj
  Q[T_1,\dots, T_c]$, $\cL = \cO_X(1)$, and $W = \sum_i f_iT_i$. Then
$$
\Nonreg(Y) = \bP^{c-1}_{Q/\fm} \subseteq \bP^{c-1}_Q.
$$
\end{ex}

\subsection{Chern classes of matrix factorizations}
If $\bE \in mf(X, \cL,
W)$ is a matrix factorization that admits a connection (for example, if $p$ is affine), we define strict morphisms 
$$
\Psi^{(j)}_{\bE, \bnabla}:
\bE \to \Omega^{(j)}_{dW} \otimes \bE, \, \,  \text{ for $j \geq 0$},
$$
recursively, by letting
$$
\Psi^{(0)} = \id_\bE, \, \,  \Psi^{(1)}_{\bE, \bnabla} = \Psi_{\bE, \bnabla},
$$
and, for $j \geq 2$, defining 
$\Psi^{(j)}_{\bE, \bnabla}$ as the composition of
$$
\bE \xra{\id_{\Omega^{(1)}} \otimes \Psi^{(j-1)}_{\bE, \bnabla}}
\Omega^{(1)}_{dW} \otimes \Omega^{(j-1)}_{dW} \otimes \bE 
\xra{\smsh \otimes \id_\bE}
\Omega^{(j)}_{dW} \otimes \bE.
$$
By Lemma \ref{lem:ind}, $\Psi^{(j)}_{\bE, \bnabla}$ is independent up to
homotopy of the choice of $\bnabla$.

\begin{ex} Continuing with Examples \ref{Ex816}, \ref{Ex816b} and
  \ref{Ex816c}, the map $\Psi^{(\smsh j)}$ arising from the trivial
  connection is in degree zero given by
$$
\overbrace{(1 + dA) \cdots 
(1 + dA) (1 + dB)}^{\text{$j$ factors}}
$$
for $j$ even and by
$$
\overbrace{(1 + dB) \cdots 
(1 + dA) (1 + dB)}^{\text{$j$ factors}}
$$
for $j$ odd,
where the product is taken in the ring $\Mat_{r \times
  r}(\Omega^\cdot_{Q/k})$.
\end{ex}

By the duality enjoyed by the category $mf(X, \cL, Q)$, we obtain a morphism of objects of $mf(X, \cL, 0)$ of the form 
$$
\tilde{\Psi}^{(j)}_{\bE}: \EndMF(\bE) = \bE^* \otimes_{mf} \bE \to \Omega^{(j)}_{dW}.
$$
Moreover, as a morphism in the 
category $[mf(X, \cL, 0)]_\naive$, it is independent of the choice of connection.
Take $j = n$ and assume $n!$ is invertible, so that we have an
isomorphism
$$
\Omega^{(n)}_{dW} \cong \fold \Omega^\cdot_{dW}.
$$
We obtain the morphism
$$
\EndMF(\bE) \to \fold \Omega^\cdot_{dW}
$$
in $mf(X, \cL, 0)$ and hence a map of unbounded complexes
\begin{equation} \label{L131}
ch_\bE: \unfold \EndMF(\bE) \to \unfold \fold \Omega^\cdot_{dW}.
\end{equation}
The identity morphism on $\bE$ determines a morphism $\cO_X \xra{\id_\bE} \unfold \EndMF(\bE)$, that is, an element of $H^0(X, \unfold \EndMF(\bE))$.

\begin{defn} Under Assumptions \ref{initialassume} and \ref{initialassume2}, define the {\em ad hoc Chern character of
    $\bE \in mf(X, \cL, W)$ relative to $p: X \to S$} to be 
$$
ch(\bE) = ch_\bE(\id_\bE) \in
\HHn_0(X/S, \cL, W).
$$
\end{defn}

Recall that when $n$ is even, we have the map 
$\HHn_0(X/S, \cL, W) \to H^0(X, \cJ_W\sst)$ defined in \eqref{E38}, and hence we obtain an
invariant of $\bE$ in $H^0(X, \cJ_W\sst)$. Since this invariant will
be of crucial importance in the rest of this paper, we give it a name:

\begin{defn} Under Assumptions \ref{initialassume} and \ref{initialassume2} and with $n$ even, define the {\em
    ad hoc top Chern class of
    $\bE \in mf(X, \cL, W)$ relative to $p: X \to S$} to be 
the element $\ctop(\bE) \in H^0(X, \cJ_W\sst)$ given as the image of
$ch(\bE)$ under \eqref{E38}.
\end{defn}

\begin{ex} \label{ex:103}
With the notation and assumptions of Example \ref{Ex816}, we have
$$
\HHn_0(X/S, \cL, W) = \HHn_0(\Spec(Q)/ \Spec(k), \cO, W) =
\bigoplus_i \frac{\ker\left(\Omega^{2i}_{Q/k} \xra{dW}\Omega^{2i+1}_{Q/k}\right)}{\im\left(\Omega^{2i-1}_{Q/k} \xra{dW}\Omega^{2i}_{Q/k}\right)}.
$$
and, when $n$ is even,  $H^0(X, \cJ_W \sst)$ is the last summand:
$$
H^0(X, \cJ_W \sst)= \frac{\Omega^n_{Q/k}}{dW \smsh
  \Omega^{n-1}_{Q/k}}.
$$
We have 
$$
\ctop(\bE) =
\frac{2}{n!} \tr(\overbrace{dA dB  \cdots dB}^{\text{$n$
    factors}}).
$$
Here $dA$ and $dB$ are $r \times r$ matrices with entries in $\Omega^1-{Q/k}$ and the product is occurring 
in the ring $\Mat_{r \times r}(\Omega^\cdot_{Q/k})$, and  the map
$\tr: \Mat_{r \times r}(\Omega^n_{Q/k}) \to \Omega^n_{Q/k}$ is the usual trace map.

Our formula for $\ctop(\bE)$ coincides with the ``Kapustin-Li'' formula \cite{KapLi} found in many other places, at least up to a sign:
In the work of Polishchuck-Vaintrob \cite[Cor 3.2.4]{PV}, for example, 
there is an additional a factor of 
$(-1)^{n \choose 2}$ in their formula for the Chern character.
\end{ex}

Since the top Chern class of $\bE \in mf(X, \cL, W)$ is defined as the image of $\id_E$ under the composition
$$
\unfold \EndMF(\bE) \xra{ch_\bE} \unfold \fold \Omega^\cdot_{dW} \to \cJ_W \sst
$$
of morphisms of complexes of coherent sheaves on $X$,
the following result is an immediate consequence of Proposition \ref{prop38}. Recall that for a quasi-coherent sheaf $\cF$ on $X$ and a closed subset $Z$ of
$X$, 
$H^0_Z(X, \cF)$ denotes the kernel of $H^0(X, \cF) \to H^0(X \setminus Z, \cF)$.

\begin{prop} \label{PropSupp} For any $\bE \in mf(X, \cL, W)$, its top Chern class is supported on $\Nonreg(Y)$:
$$
\ctop(\bE) \in H^0_{\Nonreg(Y)}(X, \cJ_W\sst).
$$
\end{prop}

\begin{rem} It is also follows from what we have established so far that the Chern character of any $\bE$ is supported on $\Nonreg(Y)$ in the sense
  that it lifts canonically to an element of 
$$
\bH^0_{\Nonreg(Y)}(X, \unfold \fold \Omega^\cdot_{dW}),
$$
where in general $\bH^0_Z$ denote local hyper-cohomology of a complex of coherent sheaves.
\end{rem}

\subsection{Functorality of the Chern character and the top Chern
  class}
The goal of this subsection is to prove Chern
character is functorial in $S$.
Throughout, we suppose Assumptions \ref{initialassume}
and \ref{initialassume2}
hold.

Let $i: S' \to S$ be a morphism of Noetherian schemes and write $\cL'$, $X'$,
$W'$ for the evident pull-backs: $X' = X \times_S S'$, $\cL'_S =
i^*(\cL_S)$, and $W' = i^*(W)$. The typical application will occur when $i$ is the
inclusion of a closed point of $S$.

There is a functor 
\begin{equation} \label{E902b}
i^*: mf(X, \cL, W) \to mf(X', \cL', W')
\end{equation}
induced by pullback along the map $X' \to X$ induced from $i$.
There is a canonical isomorphism $i^*\Omega^\cdot_{dW} \cong
\Omega^\cdot_{dW'}$ and hence an hence an induced map
$$
i^*:
\HHn_0(X/S, \cL, W) \to 
\HHn_0(X'/S', \cL', W').
$$

\begin{prop} \label{prop:natural}
For $\bE \in mf(X, \cL, W)$, we have
$$
i^*(ch(\bE)) = ch(i^*(\bE)) \in 
\HHn_0(X'/S', \cL', W').
$$
\end{prop}

\begin{proof} Recall that the Chern character of $\bE$ is determined by
  a map of matrix factorizations
$$
\bE \to \bE \otimes \Omega^\cdot_{dW}
$$
which is itself determined by a choice of connection $\cE_j \to \cE_j
\otimes_{\cO_X} \Omega^1_{X/S}$ for $j = 0,1$. Since our connections
are $\cO_S$-linear, we have the pull-back connection
$$
i^* \cE_j \to i^*(\cE_j )\otimes_{\cO_{X'}} \Omega^1_{X'/S'},  j=0 ,1.
$$
which we use to define the map
$$
i^* \bE \to i^*\bE \otimes \Omega^\cdot_{dW'}.
$$
Using also that $i^* \End(\bE) = \End(i^* \bE)$, 
it follows that the square
$$
\xymatrix{
\End(\bE) \ar[r] \ar[d] & \fold \Omega^\cdot_{dW} \ar[d] \\
\R i_* \End(i^* \bE) \ar[r]  & \R i_* \fold \Omega^\cdot_{dW'} \\
}
$$
commutes. Applying $\bH^0(X, \unfold(-))$ gives the commutative square
$$
\xymatrix{
\bH^0(X, \unfold \End(\bE)) \ar[r] \ar[d]^{i^*} & \HHn_0(X/S, \cL, W) \ar[d]^{i^*} \\
\bH^0(X', \unfold \End(i^* \bE)) \ar[r]  & \HHn_0(X'/S', \cL', W'). \\
}
$$
The result follows from the fact that
$i^*: \bH^0(X, \unfold \End(\bE)) \to
\bH^0(X', \unfold \End(i^* \bE))$ sends $\id_\bE$ to $\id_{i^*\bE}$.
\end{proof}

Recall from \eqref{E38} that, when $n$ is even, there is a natural map
$$
\HHn_0(X/S, \cL, W) \to 
\Gamma(X, \cJ_{X/S, \cL, W} \sst)
$$
Since $i^* \cJ_{X/S, \cL, W} \cong \cJ_{X'/S', \cL', W'}$, we obtain a
map 
$$
i^*: \Gamma(X, \cJ_{X/S, \cL, W} \sst)
\to
\Gamma(X', \cJ_{X'/S', \cL', W'} \sst).
$$

\begin{cor} \label{Cor130}
When $n$ is even,
for any $\bE \in mf(X, \cL, W)$, we have
$$
i^*(\ctop(\bE)) = \ctop(i^*(\bE)) \in 
\Gamma(X, \cJ_{X'/S', \cL', W'} \sst).
$$
\end{cor}

\begin{proof}
This follows from Proposition \ref{prop:natural} and the fact that
$$
\xymatrix{
\HHn_0(X/S, \cL, W) \ar[r] \ar[d]^{i^*} &  \Gamma(X, \cJ_{X/S, \cL, W}\sst)
\ar[d]^{i^*} \\
\HHn_0(X'/S', \cL', W') \ar[r] &  \Gamma(X', \cJ_{X'/S', \cL', W'} \sst)
\\
}
$$
commutes.
\end{proof}

\section{Relationship with affine complete intersections}

We present some previously known results that relate twisted matrix
factorizations with the singularity category of affine complete intersections.
We start with some basic background.

For any Noetherian scheme $X$, we write $D^b(X)$ for the bounded
derived category of coherent sheaves on $X$
 and $\Perf(X)$ for the
full subcategory of perfect complexes --- i.e., those complexes of
coherent sheaves on $X$ that are locally quasi-isomorphic to bounded
complexes of finitely generated free modules.
We define
$\Dsg(X)$ to be the Verdier
quotient $D^b(X)/\Perf(X)$. For a Noetherian ring $R$, we write $\Dsg(R)$
for $\Dsg(\Spec(R))$. 

The category $\Dsg(X)$ is triangulated. For a pair of finitely generated $R$-modules $M$ and $N$, we define their
{\em stable $\Ext$-modules} as
$$
\sExt_R^i(M,N) := \Hom_{\Dsg(R)}(M, N[i]),
$$
where on the right we are interpreting $M$ and $N[i]$ as determining
complexes (and hence objects of $\Dsg(R)$) consisting of one non-zero component, lying in degrees $0$
and $-i$, respectively. Note that since $D^b(R) \to \Dsg(R)$ is a
triangulated functor by construction and the usual $\Ext_R$-modules are
given by $\Hom_{D^b(R)}(M, N[i])$, there is a canonical map
$$
\Ext^i_R(M,N) \to \sExt_R^i(M,N)
$$
that is natural in both variables.

\begin{prop} \cite[1.3]{Buchweitz} 
For all finitely generated $R$-modules $M$ and $N$, the natural map
$$
\Ext^i_R(M,N) \to \sExt_R^i(M,N)
$$
is an isomorphism for $i \gg 0$.
\end{prop}

\subsection{Euler characteristics and the Herbrand difference for hypersurfaces}

Throughout this subsection, assume $Q$ is a regular ring and $f \in Q$ is a
  non-zero-divisor, and let $R = Q/f$. Given an object $\bE = (E_1
  \darrow{\alpha}{\beta} E_0) \in mf(Q, f)$, define $\coker(\bE)$ to
  be $\coker(E_1 \xra{\alpha} E_0)$. The $Q$-module 
  $\coker(\bE)$ is annihilated by $f$ and we will regard it as
  an $R$-module. 

\begin{thm}[Buchweitz] \cite{Buchweitz} \label{thm1030}
If $R = Q/f$ where $Q$ is a regular ring and $f$ is a
  non-zero-divisor, there is an equivalence of triangulated categories
$$
hmf(Q,f) \xra{\cong} \Dsg(R)
$$
induced by sending a matrix factorization $\bE$ to $\coker(\bE)$,
regarded as an object of $\Dsg(R)$.

In particular,  we have an isomorphisms 
$$
\Ext_{hmf}^i(\bE, \bF) := \Hom_{hmf}(\bE,
\bF[i]) \cong \sExt_R^i(\coker(\bE), \coker(\bF)), \text{ for all $i$,}
$$
and
$$
\Ext_{hmf}^i(\bE, \bF) \cong \Ext_R^i(\coker(\bE), \coker(\bF)), \text{ for $i \gg 0$.}
$$
\end{thm}

Since the category $hmf(Q_\fp, f)$ is trivial for all $f \in \fp$ such that $R_\fp$ is regular, 
$\Hom^i_{hmf}(\bE, \bF)$ is an $R$-module supported on
$\Nonreg(R)$, for all $i$.
If we assume 
$\Nonreg(R)$ is a finite set of maximal
  ideals, then $\Hom^i_{hmf}(\bE, \bF)$ has finite length as an
  $R$-module, allowing us to make the following definition.

\begin{defn} Assume $Q$ is a regular ring and $f \in Q$ is a
  non-zero-divisor such that $\Nonreg(R)$ is a finite set of maximal
  ideals, where $R := Q/f$.   For $\bE, \bF \in mf(Q,f)$, define the {\em Euler
    characteristic} of $(\bE, \bF)$ to be
$$
\chi(\bE, \bE) = \len \Hom_{hmf}^0(\bE, \bF) -\len \Hom_{hmf}^1(\bE,
\bF).
$$
\end{defn}

In light of Theorem \ref{thm1030}, if $M$ and $N$ are the cokernels of
$\bE$ and $\bF$, then
\begin{equation} \label{E22}
\chi(\bE, \bF) = h^R(M,N),
\end{equation}
where the right-hand side is the {\em Herbrand difference} of
the $R$-modules $M$ and $N$, defined as 
$$
\begin{aligned}
h^R(M,N) & = \len \sExt^0_R(M,N) - \len \sExt^1_R(M,N) \\
& = 
\len \Ext^{2i}_R(M,N) - \len \Ext^{2i+1}_R(M,N), \, i \gg 0. \\
\end{aligned}
$$

\begin{lem} \label{lem1023b}  
Suppose $\phi: Q \to Q'$ is a flat ring homomorphism between
  regular rings, $f$ 
is a non-zero-divisor in $Q$, and $f' := \phi(f)$ is a
non-zero-divisor of $Q'$. 
Assume $\Nonreg(Q/f) = \{\fm\}$ and  $\Nonreg(Q'/f') = \{\fm'\}$ for maximal ideals $\fm$, $\fm'$ satisfying the condition that  
$\fm Q'$ is $\fm'$-primary.

Then, for all $\bE, \bF \in mf(Q,f)$, 
$$
\chi_{mf(Q,f)}(\bE, \bF) =
\lambda \cdot \chi_{mf('Q,f')}(\bE \otimes_Q Q', \bF \otimes_Q Q'),
$$
where $\lambda := \len_{Q'}(Q'/\fm Q')$. 
\end{lem}

\begin{proof} Let $R = Q/f$ and $R' = Q'/f'$. 
The induced map $R \to R'$ is
  flat and hence for any pair of $R$-modules $M$ and $N$,  we have an isomorphism
$$
\Ext^i_R(M,N) \otimes_R R' \cong
\Ext^i_{R'}(M \otimes_R R',N \otimes_R R')
$$
of $R'$-modules.
For $i \gg 0$, 
$\Ext^i_R(M,N)$ is supported on $\{\fm\}$
and
$\Ext^i_{R'}(M \otimes_R R', N \otimes_R R')$ is supported on $\{\fm'\}$.
It 
follows that
$$
\len_{R'} \Ext^i_{R'}(M \otimes_R R',N \otimes_R R')
= \len_{R'} (\Ext^i_R(M,N) \otimes_R R') = \lambda \len_R \Ext^i_R(M,N).
$$
Hence $h^{R'}(M', N') = \lambda h^R(M,N)$ and
the result follows from \eqref{E22}.
\end{proof}

\subsection{Matrix factorizations for complete intersections}

Using a Theorem of Orlov \cite[2.1]{OrlovLGEquivalence}, one may generalize
Theorem \ref{thm1030} to complete intersection rings. The precise statement is the next Theorem. 
Versions of it are found in \cite{PVStack}, \cite{LinPom}, \cite{OrlovLG}, \cite{PositCoherent}; the one given here is from \cite[2.11]{BW2}.

\begin{thm} \label{thm:BW}
Assume $Q$ is a regular ring of finite Krull dimension and $f_1, \dots, f_c$ is a
  regular sequence of elements of $Q$. Let $R = Q/(f_1, \dots, f_c)$, define $X = \bP^{c-1}_Q = \Proj
  Q[T_1, \dots, T_c]$ and set $W  = \sum_i f_iT_i \in \Gamma(X, \cO(1))$.
There is an equivalence of
  triangulated categories
$$
\Dsg(R) \xra{\cong} hmf(X, \cO(1), W).
$$
\end{thm}

The isomorphism of Theorem \ref{thm:BW} has a certain naturality
property that we need. Suppose $a_1, \dots, a_c$ is a sequence of element of $Q$ that
generate the unit ideal, and let $i: \Spec(Q) \into X$ be the
associated closed immersion. 
Then we have a functor
$$
i^*: hmf(X, \cO(1), W) \to hmf(Q, \sum_i a_i f_i).
$$
Technically, the right hand side should be
$hmf(\Spec(Q), i^*\cO(1), i^*(W))$, but it is canonically isomorphic
to $hmf(Q, \sum_i a_i f_i)$ because there is a canonical isomorphism 
$i^* \cO(1) \cong Q$ that sends $i^*(W)$ to $\sum_i a_i f_i$.

Also, the quotient map $Q/(\sum_i a_i f_i) \onto R$ has
finite projective dimension, since,  locally on $Q$, it is given by modding out
by a regular sequence of length $c-1$. We thus an induced functor
$$
\res: \Dsg(R) \to \Dsg(Q/\sum_i a_i f_i) 
$$
on singularity categories, 
given by restriction of scalars.

\begin{prop} \label{prop21}
With the notation above,
the square 
$$
\xymatrix{
\Dsg(R) \ar[r]^{\cong \phantom{XXX}} \ar[d]^{\res} & 
hmf(\bP^{c-1}_Q, \cO(1), W) \ar[d]^{i^*} \\
\Dsg(Q/\sum_i a_if_i) \ar[r]^\cong & 
hmf(Q, \sum_i a_i f_i) 
}
$$
commutes up to natural isomorphism. 
\end{prop}

\begin{proof} 
To prove this, we need to describe the equivalence 
of Theorem \ref{thm:BW} explicitly. Let $Y \subseteq \bP^{c-1}_Q$ denote the closed subscheme cut out by $W$. Then there are a pair of
equivalences
\begin{equation} \label{E25b}
\Dsg(R) \xra{\cong} \Dsg(Y) \xla{\cong}  hmf(X, \cO(1), W),
\end{equation}
which give the equivalence $\Dsg(R) \xra{\cong} hmf(X, \cO(1), W)$ of the Theorem.

The left-hand equivalence of \eqref{E25b} was established by Orlov \cite[2.1]{OrlovLGEquivalence}
(see also \cite[A.4]{BW2}), and it is induced by the 
functor sending a bounded complex of finitely generated $R$-modules $M^\cdot$ to $\beta_*\pi^*(M^\cdot)$, where $\pi: \bP^{c-1}_R \to \Spec(R)$ is the evident
projection and $\beta: \bP^{c-1}_R \into Y$ is the evident closed immersion. 

The right-hand equivalence of \eqref{E25b} is given by \cite[6.3]{BW1}, and 
it sends a twisted matrix factorization $\bE = (\cE_1
\xra{\alpha} \cE_0 \xra{\beta} \cE_1(1))$ to $\coker(\alpha)$ (which may be regarded as a coherent sheaf on $Y$.)

For brevity, set $g = \sum_i a_i f_i$.
The Proposition will follow once we establish that the diagram
\begin{equation} \label{E25}
\xymatrix{
\Dsg(R)  \ar[r]^{\beta_* \circ \pi^*}_{\cong} \ar[d]_{k_*} & \Dsg(Y) \ar[dl]^{\bL j^* \phantom{XXX}} & 
hmf(\bP^{c-1}_Q, \cO(1), W) \ar[l]_{\cong \phantom{XXX}}\ar[d]^{i^*} \\
\Dsg(Q/g) & & hmf(Q, g) \ar[ll]^\cong  \\
}
\end{equation}
commutes up to natural isomorphism, where $k: \Spec(R) \into \Spec(Q/g)$ is the canonical closed immersion and
$j: \Spec(Q/g) \to Y$ is the restriction of $i$. (The closed immersion $j$ is locally a complete intersection and thus of finite flat dimension. It follows that
it induces a morphism on singularity categories, which we write $\bL j^*$.)

The right-hand square of \eqref{E25} commutes since for a matrix factorization $\bE =
(\cE_1 \to \cE_0 \to \cE_1(1))$ we have
$$
j^* \coker(\cE_1
\to \cE_0) \cong \coker(i^* \cE_1 \to i^* \cE_0)
$$ 
and 
$$
j^* \coker(\cE_1
\to \cE_0) \cong \bL
  j^* \coker(\cE_1
\to \cE_0).
$$
The first isomorphism is evident. The second holds since $j$ has finite flat dimension and
$\coker(\cE_1 \to \cE_0)$ is an ``infinite syzygy''; that is, 
$\coker(\cE_1 \to \cE_0)$  has a right resolution by locally free sheaves on $Y$, namely
$$
\gamma^* \cE_1(1) \to \gamma^* \cE_0(1) \to \gamma^* \cE_1(2) \to \cdots,
$$
where $\gamma: Y \into X$ is the canonical closed immersion.

To show the left-hand triangle commutes, we consider the Cartesian square of closed immersions
$$
\xymatrix{
\bP^{c-1}_{R} \ar@{>->}[r]^{\beta}  & Y \\
\Spec(R) \ar@{>->}[r]^{k} \ar@{>->}[u]_{i_R}  & \Spec(Q/g) \ar@{>->}[u]^{j}, \\
}
$$
where $i_R$ denote the restriction of $i$ to $\Spec(R)$.
This square is $\Tor$-independent (i.e.,
$\uTor_i^{\cO_{Y}}(\beta_* \cO_{\bP^{c-1}_{R}}, j_* \cO_{\Spec(Q/g)}) = 0$ for
$i > 0$),
from which it follows that
$$
\bL j^* \circ \beta_* = k_* \circ \bL i_R^*.
$$
Since $\bL i_R^* \circ \pi^*$ is the identity map, the left-hand triangle of \eqref{E25} commutes.
\end{proof}

\section{Some needed results on the geometry of complete intersections}

Let $k$ be a field and $Q$ a smooth $k$-algebra of dimension $n$. Set $V = \Spec(Q)$, a smooth affine $k$-variety. Given $f_1, \dots, f_c \in Q$, we have the
associated morphism
$$
\vf := (f_1, \dots, f_c): V \to \A^c_k
$$
of smooth $k$-varieties, induced by the $k$-algebra map $k[t_1, \dots, t_c] \to Q$ sending $t_i$ to $f_i$.
Define the {\em Jacobian} of $\vf$, written $J_\vf$,  to be the map  
$$
J_\vf: Q^c \to \Omega^1_{Q/k}
$$ 
given by $(df_1, \dots, df_c)$. (More
formally, $J_\vf$ is the map $\vf^* \Omega^1_{\A^c_k/k} \to \Omega^1_{V/k}$
induced by $\vf$, but we identify $\vf^* \Omega^1_{\A^c_k/k}$ with $Q^c$
by using the basis $dt_1, \dots, dt_c$ of $\Omega^1_{\A^c_k/k}$.
Also, $J_\vf$ is really the dual of
what is often called the Jacobian of $\vf$.)

For example, if $Q = k[x_1, \dots, x_n]$, then $J_\vf$ is given by the $n \times c$
matrix $(\partial f_i/\partial x_j)$, using the basis $dx_1, \dots, dx_n$ of $\Omega^1_{Q/k}$.

For a point $x \in V$, let $\kappa(x)$ denote its residue field
and define 
$$
J_\vf(x) = J_\vf \otimes_Q \kappa(x): \kappa(x)^c \to
\Omega^1_{Q/k} \otimes_Q \kappa(x) \cong \kappa(x)^n
$$
to be the map on finite dimensional $\kappa(x)$-vector spaces
induced by $J_\vf$. 
Define 
\begin{equation} \label{def1127}
V_j := \{ x \in V \, | \, \rank J_\vf(\kappa(x))  \leq j \}
\subseteq V,
\end{equation}
so that we have a filtration
$$
\emptyset = V_{-1} \subseteq V_0 \subseteq  \cdots \subseteq V_c = V.
$$
Note that the set  $V_{c-1}$ is the singular locus of the map $\vf$.

For example, if $Q = k[x_1, \dots, x_n]$, then $V_j$ is defined by the vanishing of the
$(j+1) \times (j+1)$ minors of the matrix
$(\partial f_i/\partial x_j)$.

\begin{ex} Let $k$ be a field of characteristic not equal to $2$, let
  $Q = k[x_1, \dots, x_n]$ so that $V = \A^n_k$,  let $c = 2$,   and
  define $f_1 = \frac{1}{2}(x_1^2 +
  \cdots + x_n^2)$ and $f_2 = \frac12(a_1 x_1^2 + \cdots + a_n x_n^2)$, where $a_1,
  \dots, a_n \in k$ are such that $a_i \ne a_j$ for all $i \ne j$. Thus $\vf$ is a morphism
  between two affine spaces: $\vf: \A^n_k \to \A^2_k$.
For the usual basis,  the Jacobian is the matrix
$$
J_\vf =
\begin{bmatrix}
x_1 & a_1x_1 \\
x_2 & a_2x_2 \\
\vdots & \vdots \\
x_n & a_nx_n \\
\end{bmatrix}.
$$
Then $V_0$ is just the origin in $V = \A^n_k$ and 
$V_1$ is the union of the coordinate axes. 
\end{ex}

In the previous example, 
we have $\dim(V_j) \leq
j$ for all $j < c$. This will turn out to be a necessary assumption 
for many of our result.

Let us summarize the notations and assumptions we will need in much of the rest of this paper:

\begin{assumptions} \label{assume} Unless otherwise indicated, from now on we assume:
\begin{enumerate}

\item $k$ is field.

\item $V = \Spec(Q)$ is smooth affine $k$-variety of dimension $n$.

\item $f_1, \dots, f_c \in Q$ is a
regular sequence of elements. We write
$$
\vf: V \to \A^c_k
$$ 
for the flat morphism given by $(f_1, \dots, f_c)$.
(Note that $c \leq n$.)

\item The singular locus of $\vf^{-1}(0) = \Spec(Q/(f_1, \dots, f_c))$, is zero-dimensional --- say $\Sing(\vf^{-1}(0)) = \{v_1, \dots, v_m\}$ where the $v_i$'s are closed points of $V$.

\item $\dm(V_j) \leq j$ for all $j < c$, where $V_j$ is defined in \eqref{def1127}. \label{L1015}

\end{enumerate}
\end{assumptions}

Observe that the singular locus of $\vf^{-1}(0)$ may be identified with $V_{c-1} \cap \vf^{-1}(0)$. 

For the results in this paper, we will be allowed to shrink about the singular locus of $\vf^{-1}(0)$. In this situation, 
if $\chr(k) = 0$,  the last assumption in above list is unnecessary, as we now prove:

\begin{lem} \label{lem:103} If all of Assumptions \ref{assume} hold
  except possibly \eqref{L1015} and $\chr(k) = 0$,  then for a
suitably small affine open neighborhood $V'$ of the set $\{v_1, \dots, v_m\}$, 
we have $\dim(V'_j) \leq i$ for all
 $j \leq c-1$. 
\end{lem}

\begin{proof} Since the induced map $\Sing(\vf) \to \A^c$ has finite type,
the set 
$$
B := \{y \in \Sing(\vf) \, | \, \Sing(\vf) \cap \dim(\vf^{-1}(\vf(y))
> 0  \}
$$ 
is a
  closed subset of $\Sing(\vf)$ by the upper-semi-continuity of fiber dimensions
  [EGAIV, 13.1.3]. 
Since 
$\Sing(\vf^{-1}(0)) = \vf^{-1}(0) \cap \Sing(\vf)$ is a finite set of closed points,
it follows that $\{v_1, \dots, v_m\} \cap B = \emptyset$. Thus
$V' := V \setminus B$ is an open neighborhood of the $v_i$'s and $\Sing(\vf|_{V'}) = \Sing(\vf)
  \cap V' = V'_{c-1}$ is quasi-finite over $\A^c$. Shrinking $V'$ further, we may assume it is affine.

By \cite[III.10.6]{Hartshorne}, we have
$\dim(\overline{\vf(V'_j)}) \leq j$ for all $j$. Since 
$V'_{c-1} \to \overline{\vf(V'_{c-1})}$ is quasi-finite, so is 
$V'_j \to \overline{\vf(V'_j)}$ for all $j \leq c-1$, and hence 
$\dim(V'_j) \leq i$ for all
 $j \leq c-1$.
\end{proof}

\begin{ex} Suppose $\chr(k) = p > 0$, $Q = k[x_1, \dots, x_c]$,  and let $f_1(x) =
  x_1^p, \cdots, f_c(x) = x_c^p$. Then $J_\vf: Q^c \to \Omega^1_{Q/k}
  \cong Q^c$ is the zero map,  and yet $k[x_1, \dots, x_c]/(x_1^p,
  \dots, x_c^p)$ has an isolated singularity at $(x_1, \dots, x_c)$. This
  shows that the characteristic $0$ hypothesis is necessary in Lemma \ref{lem:103}.
\end{ex}

Given the set-up in Assumptions \ref{assume}, we let 
$$
S  =  \bP^{c-1}_k  = \Proj k[T_1, \dots, T_c] 
$$
and
$$
X = \bP^{c-1}_k \times_k V = \Proj Q[T_1, \dots, T_c], 
$$
and we define 
$$
W = f_1T_1 + \cdots + f_c T_c \in \Gamma(X, \cO(1)).
$$
As before, we have a map
$$
dW: \cO_{X} \to 
\Omega^1_{X/S}(1),
$$
or, equivalently, a global section 
$dW \in \Gamma(X, \Omega^1_{X/S}(1))$.
Recall that $\Omega^1_{S/X}(1)$ is the coherent sheaf associated to the 
graded module $\Omega^1_{Q/k}[T_1, \dots, T_c](1)$,
and $dW$ may be given explicitly 
as the degree one element 
$$
dW = df_1 T_1 + \cdots +
df_c T_c \in  \Omega^1_{Q/k}[T_1, \dots, T_c],
$$
or, in other words,
\begin{equation} \label{E1127}
dW = J_\vf \cdot \vecfour{T_1}{T_2}{\vdots}{T_c}.
\end{equation}

\begin{prop} \label{P1}
With Assumptions \ref{assume},  $dW$ is a regular section of
  $\Omega^1_{X/S}(1)$. 
\end{prop}

\begin{proof}
 Since $X = \bP^{c-1}_k \times_k V$ is smooth of dimension $n + c- 1$ and
  $\Omega^1_{X/S}$ is locally free of rank $n$,
it suffices to prove the subscheme
$Z$ cut out by $dW$ has dimension at most $c-1$. 
Using \eqref{E1127}, we see that 
the fiber of $Z \to V$ over a  point $x \in V$ is a linear subscheme of
$\bP^{c-1}_{\kappa(x)}$ of dimension $c - 1 - \rank(J_\vf(x))$. 
That is, if $x \in V_j \setminus V_{j-1}$, then the fiber over $Z \to
Y$ over $y$ has dimension $c-1-j$. (This includes the case $j = c$ in
the sense that
the fibers over $V_c \setminus V_{c-1}$ are empty.) Since $\dim(V_j)
\leq j$ if $j < c$ by assumption,  it follows
that $\dim(Z) \leq c-1$. 
\end{proof}

\begin{cor} \label{C1}
If Assumptions \ref{assume} hold, then the Jacobian complex $\Omega^\cdot_{dW}$
is exact everywhere except in the right-most position, and hence
determines a resolution of $\cJ_W(n)$ by locally free coherent
sheaves. 
\end{cor}

We will need the following result in the next section.

\begin{prop} \label{prop:iso}
Under Assumptions \ref{assume},  
  there is an open dense subset $U$ of $\bP^{c-1}_k$ such that for
  every $k$-rational point 
  $[a_1: \cdots : a_c] \in U$, the singular locus of the morphism
$$
\sum_i a_i f_i: V \to \A^1_k
$$
has dimension $0$.
\end{prop}

\begin{proof} 
Let $Z$ be the closed subvariety of $\bP^{c-1}_k \times_k V$ defined by the vanishing of $dW$ (see \eqref{E1127}). 
The fiber of the projection map $\pi: Z \to \bP^{c-1}_k$  over a $k$-rational point $[a_1: \cdots: a_c]$ may be identified with 
the singular locus of the morphism $\sum_i a_i f_i: V \to \A^1_k$. 
The claim is thus that there is a open dense subset $U$ of $\bP^{c-1}_k$ over which the fibers of $\pi$ have dimension zero. 
Such a $U$ exists because $\dm(Z) \leq c-1$, as we established in the proof of Proposition \ref{P1}. ($U$ may be taken to be the complement of $\overline{\pi(B)}$ where $B$ is the
closed subset $\{z \in Z \, | \, \dm \pi^{-1}\pi(z) \geq 1\}$ of $Z$.)
\end{proof}

\section{The vanishing of Chern characters}

In this section, we prove that the top Chern class of a twisted matrix factorization vanishes in certain situations. We combine this with a Theorem of
Polishchuk-Vaintrob to establish the vanishing of $h_c$ and $\eta_c$ under suitable hypotheses.

\subsection{The vanishing of the top Chern class}

The following theorem forms the key technical result of this paper.

\begin{thm} \label{ctopvanishes}
If Assumptions \ref{assume} hold and in addition $n$ is even, $\chr(p) \nmid n!$,  and $c \geq 2$, then 
$\ctop(\bE) = 0$
for any $\bE
\in mf(\bP^{c-1}_Q, \cO(1), \sum_i f_iT_i)$.
\end{thm}

\begin{proof} The key point is that, 
under Assumptions \ref{assume} with $c \geq 2$, we have
\begin{equation} \label{key228}
H^0_{\bP^{c-1} \times \{v_1, \dots, v_m\}}(\bP^{c-1}_k \times_k V, \cJ_W(i))= 0
\end{equation}
for all $i < n$.

To prove this, since
$$
H^0_{\bP^{c-1} \times \{v_1, \dots, v_m\}}(\bP^{c-1} \times_k V, \cJ_W(i))
=
\bigoplus_{i=1}^m H^0_{\bP^{c-1} \times \{v_i\}}(\bP^{c-1} \times_k V, \cJ_W(i))
$$
and since we may replace  
$V = \Spec(Q)$ by any open neighborhood of $v_i$ without
  affecting the $Q$-module $H^0_{\bP^{c-1} \times \{v_i\}}(\bP^{c-1} \times_k V, \cJ_W(i))$,  we may assume
$m = 1$ and that there 
  exists a regular sequence of elements  
$x_1, \dots, x_n \in \fm$  that generate the maximal ideal $\fm$ of $Q$ corresponding to $v = v_1$.  

Write $\cC := \cC(x_1, \dots,
x_n)$ for the ``augmented Cech complex''
$$
Q \to \bigoplus_i Q\adi{x_i}
\to \bigoplus_{i,j} Q\adi{x_ix_j}
\to \cdots \to Q\adi{x_1 \cdots x_n},
$$
with $Q$ in cohomological degree $0$. Equivalently $\cC = \cC(x_1) \otimes_Q \cdots \otimes_Q \cC(x_n)$, where $\cC(x_i) = (Q \into Q\adi{x_i})$.
For any coherent sheaf $\cF$ on $\bP^{c-1}_Q$, we have 
\begin{equation} \label{E1127b}
H^0_{\bP^{c-1} \times \{v\}}(\bP^{c-1} \times_k V, \cF) = \Gamma(\bP^{c-1} \times_k V, \cH^0(\cF \otimes_Q \cC)).
\end{equation}

The complex $\cC$ is exact in all
degrees except in degree $n$, and we 
set $E = H^n(\cC)$. Explicitly, 
$$
E = \frac{Q\adi{x_1 \cdots x_n}}{\Sigma_i Q\adi{x_1 \cdots
    \hat{x_i} \cdots x_n} }.
$$
(The localization of $E$ at $\fm$ is an injective hull of the residue
field, but we do not need this fact.)
Thus we have a quasi-isomorphism $\cC \xra{\sim} E[-n]$.
 
By Corollary \ref{C1}, there is a quasi-isomorphism
$$
\Omega_{dW}^{\cdot}(-n)[n] \xra{\sim} \cJ_W.
$$
Since $\cC$ is a complex of flat modules, the map
$$
\Omega_{dW}^{\cdot}(-n)[n] \otimes_Q \cC
\xra{\sim} \cJ_W \otimes_Q \cC
$$
is also a quasi-isomorphism.
Combining this with the quasi-isomorphism $\cC \xra{\sim} E[-n]$
gives an isomorphism in the derived
category
$$
\Omega_{dW}^{\cdot}(-n) \otimes_Q E
\cong
\cJ_W \otimes \cC.
$$
For any $i$ we obtain an isomorphism
$$
\cH^0 \left(\cJ_W(i) \otimes \cC\right)
\cong
\cH^0\left(\Omega_{dW}^{\cdot}(i-n)  \otimes_Q E\right) 
$$
of quasi-coherent sheaves.

We thus obtain from \eqref{E1127b} the
isomorphism
\begin{equation} \label{E212b}
\Gamma_{\bP^{c-1}_k \times \{v\}}\left(\bP^{c-1}_k \times_k V, \cJ_W(i)\right)
\cong
\Gamma\left(\bP^{c-1}_k \times_k V, 
\cH^0\left(\Omega_{dW}^{\cdot}(i-n)  \otimes_Q E\right)\right) 
\end{equation}
But $\cH^0\left(\Omega_{dW}^{\cdot}(i-n) \otimes_Q E\right)$ is the kernel of
$$
\cO_{\bP^{c-1}_Q}(i-n) \otimes_Q E  \xra{dW \smsh -}
\Omega^1_{\bP^{c-1}_Q/\bP^{c-1}_k}(i-n+1) \otimes_Q E.
$$
Since $c-1 \geq 1$, the coherent sheaf 
$\cO_{\bP^{c-1}_Q}(i-n)$ has no global section when $i < n$, and hence
$$
\Gamma\left(\bP^{c-1}_k \times_k V, \cH^0\left(\Omega_{dW}^{\cdot}(i-n) \otimes_Q E\right)\right) = 0,
$$
which establishes \eqref{key228}.

Let $Y$ be the closed subscheme of $X = \bP^{c-1}_k \times V$ cut out by $W = \sum_i f_i T_i$. 
We claim $\Sing(Y/k) \subseteq \bP^{c-1}_k \times \{v_1, \dots, v_m\}$. Indeed,
$\Sing(Y/k)$ is defined by the equations $\sum_i d(f_i) T_i = 0$ and $f_i = 0$, $i = 1, \dots c$, where $d(f_i) \in \Omega^1_{Q/k}$. 
The first equation cuts out a subvariety contained in $\bP^{c-1}_k \times_k V_{c-1}$, where, recall, $V_{c-1}$ denotes the singular locus of $\vf: V \to
\A^c_k$. The equations $f_i = 0$, $i = 1, \dots c$, cut out $\bP^{c-1}_k \times_k \vf^{-1}(0)$. 
But $V_{c-1} \cap \vf^{-1}(0) = \Sing(\vf^{-1}(0)) = \{v_1, \dots, v_m\}$.

Since $\Nonreg(Y) \subseteq \Sing(Y/k)$, we deduce from \eqref{key228} that
$$
H^0_{\Nonreg(Y)}(\bP^{c-1}_k \times_k V, \cJ_W(i))= 0.
$$
The vanishing of $\ctop(\bE)$ now follows from
Proposition \ref{PropSupp}.
\end{proof}

\begin{cor} \label{ctopvanishescor} 
If
\begin{itemize}
\item $k$ is a field of characteristic $0$, 

\item $Q$ is a smooth $k$-algebra of even dimension,

\item $f_1, \dots, f_c \in Q$ forms a
regular sequence of elements with $c \geq 2$, and

\item the singular locus of $Q/(f_1, \dots, f_c)$, is zero-dimensional,
\end{itemize}
then $\ctop(\bE) = 0$ for all $\bE \in mf(\bP^{c-1}_Q, \cO(1), \sum_i f_i T_i)$. 
\end{cor}

\begin{proof} Let $\{v_1, \dots, v_m\}$  be the singular locus of 
$Q/(f_1, \dots, f_c)$ and let 
$V' = \Spec(Q')$ be an affine open neighborhood of it 
such that $\dm(V'_i) \leq i$ for all $i \leq c-1$. The existence of such a 
  $V'$ is given by Lemma \ref{lem:103}.  
The result follows from the (proof of the) Theorem, since the map
$$
H^0_{\bP^{c-1}_k \times \{v_1, \dots, v_m\}}(\bP^{c-1}_Q, \cO, \sum_i f_i T_i) \to H^0_{\bP^{c-1}_k \times \{v_1, \dots, v_m\}}(\bP^{c-1}_{Q'}, \cO, \sum_i f_i
T_i) 
$$
is an isomorphism. 
\end{proof}

\subsection{The Polishchuk-Vaintrob Riemann-Roch Theorem}
We recall a Theorem of Polishchuk-Vaintrob that relates the Euler characteristic of (affine) hypersurfaces with isolated singularities to Chern characters. This
beautiful theorem should be regarded as a form of a ``Riemann-Roch'' theorem. Since Polishchuk-Vaintrob work in a somewhat different setting that we do, we begin by recalling their
notion of a Chern character.

For a field $k$ and integer $n$, define $\hQ  := k[[x_1, \dots, x_n]]$, 
a power series ring in $n$ variables,
and  let $f \in \hQ$ be such that the only
non-regular point of 
$\hQ/f$ is its maximal ideal.
Define 
$\hOmega^\cdot_{\hQ/k}$ to be the exterior algebra on the free $\hQ$-module
$$
\hOmega^1_{\hQ/k} = \hQ dx_1 \oplus  \cdots \oplus \hQ dx_n.
$$
In other words, 
$$
\hOmega^\cdot_{\hQ/k} = \Omega^{\cdot}_{k[x_1, \dots, x_n]/k}
\otimes_{k[x_1, \dots, x_n]} \hQ.
$$
Finally, 
define
$$
\hcJ(\hQ/k, f) = \coker\left( \hOmega^{n-1}_{\hQ/k}   \xra{df \smsh -}
  \hOmega^n_{\hQ/k}\right).
$$
Note that 
$$
\hcJ(\hQ/k, f) \cong \frac{k[[x_1, \dots, x_n]]}{\langle \pd{f}{x_1}, \dots, \pd{f}{x_n} \rangle} dx_1 \smsh \dots \smsh dx_n.
$$

\begin{defn} With the notation above, 
given a matrix factorization  $\bE \in mf(\hQ,f)$, upon choosing bases, it may be written as 
$\bE = (\hQ^r \darrow{A}{B} \hQ^r)$ for $r \times r$ matrices $A, B$
with entries in $\hQ$.  When $n$ is even, we define the {\em Polishchuk-Vaintrob Chern character} of $\bE$ to be 
$$
\cPV(\bE) = (-1)^{n \choose 2} tr\left(\overbrace{dA \smsh dB \smsh \cdots \smsh dB}^{\text{$n$ factors}} \right) \in \hcJ(\hQ/k, f).
$$
\end{defn}

\begin{thm}[Polishchuk-Vaintrob] \label{thm:PV}
Assume $k$ is a field of
  characteristic $0$ and  $f \in
\hQ =   k[[x_1, \dots, x_n]]$ is a power series such that $R := \hQ/f$ has
in an isolated singularity (i.e., $R_\fp$ is regular for all $\fp \ne \fm$).

If $n$ is odd, then $\chi(\bE, \bF) = 0$ for all 
$\bE, \bF
 \in mf(\hQ, f)$.

If $n$ is even, there is a pairing $\langle -,- \rangle$
on $\hcJ(\hQ, f)$, namely the
residue pairing, such that 
for all $\bE, \bF  \in mf(\hQ, f)$,  we have
$$
\chi(\bE, \bF) = \langle \cPV(\bE), \cPV(\bF) \rangle.
$$
In particular, $\chi(\bE, \bF) = 0$ if either $\cPV(\bE) = 0$ or
$\cPV(\bF) = 0$.
\end{thm}

It is easy to deduce from their Theorem the following slight generalization of it, which we will need to prove Theorem \ref{MainThm} below.

\begin{cor} \label{cor:vanish}
Assume $k$ is a field of characteristic $0$, $Q$ is a smooth $k$-algebra of dimension $n$, $f \in Q$ is a non-zero-divisor such that the singular locus of
$R := Q/f$ is zero-dimensional. Given $\bE \in mf(Q, f)$, 
if either $n$ is odd or $n$ is even and $\ctop(\bE) = 0$ in $\cJ(Q/k,f) = \Omega^n_{Q/k}/df \smsh \Omega^{n-1}_{Q/k}$, then
$$
\chi(\bE, \bF) = 0
$$
for all $\bF \in mf(Q, f)$.
\end{cor}

\begin{proof} We start be reducing to the case where $R$ has just one singular point $\fm$ and it is $k$-rational.
We have
$$
\chi_{mf(Q,f)}(\bE, \bF) = \sum_{\fm} \chi_{mf(Q_\fm, f)}(\bE_\fm, \bF_\fm)
$$
where the sum ranges over the singular points of $R$. 
If $\ok$ is an algebraic closure of $k$, then for each such $\fm$ we have
$$
\Hom_{hmf(Q_\fm,f)}^*(\bE_\fm, \bF_\fm) \otimes_k \ok \cong
\Hom_{hmf(Q_\fm \otimes_k \ok,f \otimes 1)}^*(\bE_\fm \otimes_k \ok, \bF_\fm \otimes_k \ok)
$$
and hence
$$
\chi_{mf(Q_\fm \otimes_k \ok, f \otimes 1)}(\bE_\fm \otimes_k \ok, \bF_\fm \otimes_k \ok) = [Q/\fm: \ok] \cdot \chi_{mf(Q_\fm,f)}(\bE_\fm, \bF_\fm). 
$$
It thus suffices to prove the Corollary for 
a suitably small affine open neighborhood of each maximal ideal $\fm$ of $Q \otimes_k \ok$ at which $Q \otimes_k \ok/f$ is singular.

In this case, a choice of a regular sequence of generators $x_1, \dots, x_n$ of the maximal ideal of $Q_\fm$ allows us to identify  
the completion of $Q$ along $\fm$ with $\hat{Q} = k[[x_1, \dots, x_n]]$.
By Lemma \ref{lem1023b}, 
$$
\chi_{mf(Q, f)}(\bE, \bF) =  \chi_{mf(\hQ,f)}(\hat{\bE}, \hat{\bF})
$$
where $\hat{\bE} = \bE \otimes_Q \hQ$ and $\hat{\bF} = \bF \otimes_Q \hQ$. 
Moreover, under the evident map $\cJ(Q_\fm/k, f) \to \hcJ(\hQ/k, f)$,
$\ctop(\bE)$ is sent to a non-zero multiple of  $\cPV(\hat{\bE})$ by Example \ref{ex:103}. The result now
follows from the Polishchuk-Vaintrob Theorem.
\end{proof}

\subsection{Vanishing of the higher $h$ and $\eta$ Invariants}

We apply our vanishing result for $\ctop$ to prove that the invariants $h_c$ and $\eta_c$  vanish for codimension $c \geq  2$ isolated singularities
in characteristic $0$, under some mild additional hypotheses.

First, we use the following lemma to prove the two invariants are closely related. I thank Hailong Dao for providing me with its proof.

\begin{lem}[Dao] \label{DaoLemma}
Let $R = Q/(f_1, \dots, f_c)$ for a regular ring $Q$ and a regular sequence of elements $f_1, \dots, f_c \in Q$  with $c \geq 1$. 
Assume $M$ and $N$ are finitely generated $R$-modules and that $M$ is MCM. If $\Tor_i^R(M,N)$ and $\Ext^i_R(\Hom_R(M,R), N)$ have finite length for $i \gg 0$, then
$$
\eta_c^R(M,N) = (-1)^c h_c^R(\Hom_R(M,R), N).
$$ 

In particular, if the non-regular locus of $R$ is zero dimensional and 
$h_c$ vanishes for all pairs of finitely generated $R$-modules, then so does $\eta_c$.
\end{lem}

\begin{rem} 
The conditions that $\Tor_i^R(M,N)$ has finite length for $i \gg 0$ and that $\Ext^i_R(\Hom_R(M,R), N)$ has finite length for $i \gg 0$ are
  equivalent; see \cite[4.2]{HJ}.
\end{rem}

\begin{proof} We proceed by induction on $c$. The case $c = 1$ follows from the isomorphism \cite[4.2]{DaoObservations}
$$
\sExt^i_R(\Hom_R(M,R),N)  \cong \sTor_{-i-1}^R(M, N).
$$

Let $S = Q/(f_1, \dots, f_{c-1})$ so that $R = S/(f_{c})$ with $f_c$ a   non-zero-divisor of $S$. By \cite[4.3]{DaoAsymptotic} and \cite[3.4]{CDAsymptotic},
$\Tor^S_i(M,N)$ and $\Ext^i_S(\Hom_R(M,R),N)$ have finite length for $i \gg 0$ and 
$$
\begin{aligned}
\eta_c^R(M,N) & = \frac{1}{2c} \eta_{c-1}^{S}(M, N) \\
h_c^R(\Hom_R(M,R),N) & = \frac{1}{2c} h_{c-1}^S(\Hom_R(M,R), N). \\
\end{aligned}
$$

Let $S^n \onto M$ be a surjection of $S$-modules with kernel $M'$ so that we have the short exact sequence
$$
0 \to M' \to S^n \to M \to 0.
$$
Observe that $M'$ is an MCM $S$-module. Since $f_c \cdot M = 0$ and $f_c$ is a non-zero-divisor of $S$, we have
$\Hom_S(M, S^n) = 0$.  Using also the isomorphism
$\Hom_R(M,R) \xra{\cong} \Ext_S^1(M, S)$ of $S$-modules 
coming from the long exact sequence of $\Ext$ modules associated to the sequence $0 \to S \xra{f_c} S \to R \to 0$, we obtain the short exact sequence
$$
0 \to S^n \to \Hom_S(M', S) \to \Hom_R(M,R) \to 0.
$$

In particular, it follows from these short exact sequences that $\Tor_i^S(M',N)$ and $\Ext^i_S(\Hom_S(M',S), N)$ have finite length for $i \gg 0$.  
Moreover, since $h_{c-1}^S$ vanishes on free $S$-modules and is additive for short exact sequences,  we have
$$
h_{c-1}^S(\Hom_S(M', S),N)  = h_{c-1}^S(\Hom_R(M,R), N).
$$
Using the induction hypotheses, we get 
$$
h_c^R(\Hom_R(M,R),N) = \frac{1}{2c} h_{c-1}^S(\Hom_S(M', S),N)  = \frac{1}{2c} (-1)^{c-1} \eta_{c-1}^S(M', N).
$$
Finally, $\eta_{c-1}^S(M', N) = -\eta_{c-1}^S(M,N)$, since $M'$ is a first syzygy of $M$. Combining these equations gives
$$
\eta_c^R(M,N) = (-1)^c h^R_{c-1}(\Hom_R(M,R), N).
$$

The final assertion holds since $\eta_c$ is completely determined by its values on MCM modules.
\end{proof}

\begin{thm} \label{MainThm}
Assume 
\begin{itemize}
\item $k$ is field of characteristic $0$,
\item $Q$ is a smooth $k$-algebra,
\item $f_1, \dots, f_c \in Q$ form a
regular sequence of elements, 
\item the singular locus of $R := Q/(f_1, \dots, f_c)$ is zero-dimensional, and
\item $c \geq 2$.
\end{itemize}
Then  $h^R_c(M,N) = 0$ and $\eta^R_c(M,N) = 0$
for all finitely generated
  $R$-modules $M$ and $N$.
\end{thm}

\begin{rem} \label{rem1014}
Since $k$ is a perfect field, the hypotheses that $Q$ is a smooth $k$
algebra is equivalent to $Q$ being a finitely
generated and regular $k$-algebra. Likewise, $\Sing(R/k) = \Nonreg(R)$.
\end{rem}

\begin{proof} By Lemma \ref{DaoLemma}, it suffices to prove
$h_c^R(M, N) = 0$. The value of this invariant is unchanged if we semi-localize at $\Nonreg(R)$. 
So, upon replacing $\Spec(Q)$ with a suitably small affine open neighborhood of $\Nonreg(R)$, Lemma \ref{lem:103} allows us to assume all of
Assumptions \ref{assume} hold.

Let $\bE_M, \bE_N \in mf(\bP^{c-1}_Q, \cO(1), W)$ be twisted matrix factorizations corresponding to the classes of $M , N$ in $\Dsing(R)$ under the equivalence
of Theorem \ref{thm:BW}.
Now choose $a_1, \dots, a_c \in k$ as in Proposition
  \ref{prop:iso} so that the singular locus of $Q/g$ is zero dimensional, where we define  $g := \sum_i
  a_i f_i$.  

A key fact we use is that $h_c^R$ is related to the classical Herbrand difference of the hypersurface $Q/g$:
By \cite[3.4]{CDAsymptotic}, for all finitely generated $R$-modules $M$ and $N$, we have
$$
h_c^R(M,N) = h_1^{Q/g}(M,N) = \frac12 h^{Q/g}(M,N),
$$ 
and thus it suffices to prove $h^{Q/g}(M,N) = 0$.
By  Proposition \ref{prop21} the affine matrix factorizations $i^* \bE_M, i^* \bE_N \in hmf(Q, g)$ represent the classes of $M, N$ in $\Dsing(Q/g)$,
where $i: \Spec(Q) \into \bP^{c-1}_Q$ is the closed immersion associated to the $k$-rational point $[a_1: \cdots, a_c]$ of $\bP^{c-1}_k$.
It thus follows from \eqref{E22} that 
$$
h^{Q/g}(M,N) = \chi(i^* \bE_M, i^* \bE_N).
$$
By Corollary \ref{cor:vanish}, it suffices to prove
    $\ctop(i^* \bE_M) = 0$ in $\cJ(Q_\fm/k)$ (when $n$ is even).
But  $\ctop(i^* \bE_M) = i^* \ctop(\bE_M)$ by Corollary \ref{Cor130} and $\ctop(\bE_M) = 0$ by 
Corollary \ref{ctopvanishescor}. 
\end{proof}

\appendix
\section{Relative Connections} \label{relcon}

We record here some well known facts concerning  connections for locally free
coherent sheaves. Throughout, $S$ is a Noetherian, separated scheme and
$p: X \to S$ is a smooth morphism; i.e., $p$ is separated, flat and of finite type and $\Omega^1_{X/S}$ is locally free.

\begin{defn} For a vector bundle (i.e., locally free coherent sheaf) $\cE$ on $X$, a {\em connection on
    $\cE$ relative to $p$} is a map of sheaves of abelian groups
$$
\nabla: \cE \to \Omega^1_{X/S} \otimes_{\cO_X} \cE
$$
on $X$ satisfying the Leibnitz rule on sections: given an open
  subset $U \subseteq X$ and elements $f \in \Gamma(U, \cO_X)$ and $e
  \in \Gamma(U, \cE)$, we have
$$
\nabla(f \cdot e) = df \otimes e + f \nabla(e) \, \text{ in 
$\Gamma(U, \Omega^1_{X/S} \otimes_{\cO_X} \cE)$, }
$$
where $d: \cO_X \to \Omega^1_{X/S}$ denotes exterior differentiation
relative to $p$.
\end{defn}

Note that the hypotheses imply that $\nabla$ is $\cO_S$-linear ---
more precisely, $p_*(\nabla):
p_* \cE \to p_*\left(\Omega^1_{X/S} \otimes_{\cO_X} \cE\right)$ is a
  morphism of quasi-coherent sheaves on $S$.

\subsection{The classical Atiyah class}
Let $\Delta: X \to X \times_S
X$ be the diagonal map, which, since $X \to S$ is separated,  is a closed immersion, 
and let $\cI$ denote the sheaf of
ideals cutting out $\Delta(X)$. Since $p$ is smooth, $\cI$ is locally
generated by a regular sequence. Recall that 
$\cI/\cI^2 \cong \Delta_* \Omega^1_{X/S}$.
Consider the
coherent sheaf $\ctP_{X/S} := \cO_{X \times_S X}/\cI^2$ on $X \times_S
X$. Observe that $\ctP_{X/S}$ is supported on $\Delta(X)$, so that
$(\pi_i)_* \ctP_{X/S}$ is a coherent sheaf on $X$, for $i = 1, 2$, where
$\pi_i: X \times_S X \to X$ denotes projection onto the $i$-th factor.

The  two
push-forwards $(\pi_i)_* \ctP_{X/S}$, $i = 1,2$ 
are canonically isomorphic as sheaves of abelian groups,
but have different structures as $\cO_X$-modules. 
We write $\cP_{X/S} = \cP$ for the sheaf of abelian groups
$(\pi_1)_* \ctP = (\pi_2)_* \ctP$ regarded as a $\cO_X-\cO_X$-bimodule
where  the left $\cO_X$-module structure is given by identifying it
with $(\pi_1)_* \ctP_{X/S}$ 
 the right $\cO_X$-module structure is given by identifying it
with $(\pi_2)_* \ctP_{X/S}$. 

Locally on an affine open subset $U = \Spec(Q)$ of $X$ lying over an
affine open subset $V = \Spec(A)$ of $S$, we have
$\cP_{U/V} = (Q \otimes_A Q)/I^2$, where $I = \ker(Q \otimes_A Q
\xra{- \cdot -} Q)$ and the left and right $Q$-module
structures are given in the obvious way.

There is an isomorphism of coherent sheaves on $X \times X$
$$
\Delta_* \Omega^1_{X/S} \cong \cI/\cI^2
$$ 
given locally on generators by $dg \mapsto g \otimes 1 - 1 \otimes g$. 
From this we obtain the short exact sequence
\begin{equation} \label{EA1}
0 \to \Omega_{X/S}^1 \to \cP_{X/S} \to \cO_X \to 0.
\end{equation}
This may be thought of as a sequence of $\cO_X-\cO_X$-bimodules, but
for $\Omega^1_{X/S}$ and $\cO_X$ the two structures coincide. 

Locally
on open subsets $U$ and $V$ as above, we have $\Omega^1_{Q/A} \cong I/I^2$,
and \eqref{EA1} takes the form
$$
0 \to I/I^2 \to  (Q \otimes_A Q)/I^2 \to Q \to 0.
$$

Viewing \eqref{EA1} as either a sequence of left or right modules, it
is a split exact sequence of locally free coherent sheaves on $X$. 
For example, a splitting of 
$\cP_{X/S} \to \cO_X$
as right modules may be given as follows: Recall that as a right module,
$\cP_{X/S} = (\pi_2)_* \cP$ and so a map of right modules $\cO_X \to \cP_{X/S}$ is
given by  a map $\pi_2^* \cO_X \to \cP_{X/S}$. Now,
$\pi_2^* \cO_X = \cO_{X \times_S X}$, and the map we use is the
canonical surjection.  We refer to this splitting as the 
{\em canonical right splitting} of \eqref{EA1}.

Locally on subsets $U$ and $V$ as above, the canonical right splitting of 
$(Q \otimes_A Q)/I^2 \onto Q$  is given by $q
\mapsto  1 \otimes q$.

Given a locally free coherent sheaf $\cE$ on $X$, we tensor
\eqref{EA1} on the right by $\cE$ to obtain the short exact sequence
\begin{equation} \label{EA2}
0 \to \Omega_{X/S}^1 \otimes_{\cO_X} \cE \xra{i} \cP_{X/S} \otimes_{\cO_X} \cE
\xra{\pi} \cE \to 0
\end{equation}
of $\cO_X-\cO_X$-bimodules. Taking section on affine open subsets $U$ and $V$ as
before, letting $E = \Gamma(U, \cE)$, this sequence has the form
$$
0 \to \Omega^1_{Q/A} \otimes_Q E \to (Q \otimes_A E)/I^2 \cdot E \to
E \to 0.
$$
Since \eqref{EA1} is split exact as a
sequence of right modules and tensor product preserves split exact
sequences, \eqref{EA2} is split exact as a sequence
of right $\cO_X$-modules, and the canonical right splitting of \eqref{EA1} determines a canonical right splitting of \eqref{EA2}, which we write as
$$
\can: \cE \to \cP_{X/S} \otimes_{\cO_X} \cE.
$$
The map $\can$ is given locally on sections by $e
\mapsto 1 \otimes e$.

In general, \eqref{EA2} need not split as a sequence of left
modules.
Viewed as a sequence of left modules, \eqref{EA2}
determines an element of 
$$
\Ext^1_{\cO_X}(\cE, \Omega_{X/S}^1
\otimes_{\cO_X} \cE)
\cong
H^1(X, \Omega_{X/S}^1 \otimes_{\cO_X} \cEnd_{\cO_X}(\cE)),
$$
sometimes called the ``Atiyah class'' of $\cE$ relative to $p$.
To distinguish this class from what we have
called the Atiyah class of a matrix factorization in the body of this
paper, we will call this class the {\em classical Atiyah class} of the
vector bundle $\cE$, and we write it as
$$
\cAt_{X/S}(\cE)
\in
\Ext^1_{\cO_X}(\cE, \Omega_{X/S}^1).
$$
The sequence \eqref{EA2} splits as a sequence of left modules if and
only if $\cAt_{X/S} = 0$.

\begin{lem} If $p: X \to S$ is affine, then the classical Atiyah class of any vector bundle
  $\cE$ on $X$ vanishes, and hence \eqref{EA2} splits as a sequence
  of left modules.
\end{lem}

\begin{proof} Since $p$ is affine, $p_*$ is exact. Applying $p_*$ to
  \eqref{EA2} results in a sequence of $\cO_S-\cO_S$ bimodules (which
  are quasi-coherent for both actions). But since $p \circ \pi_1 = p
  \circ \pi_2$ these two actions coincide. Moreover, since \eqref{EA2}
  splits as right modules, so does its push-forward along $p_*$. 

It thus suffices to prove the following general fact: If
$$
F := (0 \to \cF' \to \cF \to \cF'' \to 0)
$$
is a short exact sequence of vector bundles on $X$ such that $p_*(F)$
splits as a sequence of quasi-coherent sheaves on $S$, then $F$
splits. To prove this, observe that $F$ determines a class in 
$H^1(X, \cHom_{\cO_X}(\cF'', \cF'))$ and it is split if and only if
this class vanishes. We may identify $H^1(X, \cHom_{\cO_X}(\cF'', \cF'))$ 
with
$H^1(S, p_* \cHom_{\cO_X}(\cF'', \cF'))$ since $p$ is affine. Moreover,
the class of 
$F \in H^1(S, p_* \cHom_{\cO_X}(\cF'', \cF'))$
 is the image of the class of 
$p_*(F) \in 
H^1(S, \cHom_{\cO_S}(p_* \cF'', p_* \cF'))$ under the
map induced by the canonical map
$$
\cHom_{\cO_S}(p_* \cF'', p_* \cF') \to
 p_* \cHom_{\cO_X}(\cF'', \cF').
$$
But by our assumption the class of $p_*(F)$ vanishes since $p_*F$  splits.
\end{proof}

\subsection{The vanishing of the classical Atiyah class and connections}
Suppose $\sigma: \cE \to \cP_{X/S} \otimes_{\cO_X} \cE$ is a splitting
of the map $\pi$ in \eqref{EA2} as left modules and recall
$\can: \cE \to \cP_{X/S} \otimes_{\cO_X} \cE$ is the 
splitting of $\pi$ as a morphism of right modules given locally by $e \mapsto 1 \otimes e$. 
Since $\sigma$ and $\can$
are splittings of the same map regarded as a map of sheaves of abelian
groups, the difference $\sigma -
\can$ factors as $i \circ \nabla_\sigma$ for a unique map of sheaves of abelian groups 
$$
\nabla_\sigma: \cE \to \Omega^1_{X/S} \otimes_{\cO_X} \cP.
$$

\begin{lem} The map $\nabla_\sigma$ is a connection on $\cE$ relative
  to $p$.
\end{lem}

\begin{proof} The property of being a connection may be verified
  locally, in which case the result is well known. 

In more detail, restricting
  to an affine open $U = \Spec(Q)$ of $X$ lying over an affine open $V
  = \Spec(A)$ of $S$, we assume  $E$ is a projective $Q$-module and
  that we are given a splitting $\sigma$ of the map of left
  $Q$-modules $(Q \otimes_A E)/I^2 \cdot E \onto E$. The map
  $\nabla_\sigma = (\sigma -
  \can)$ lands in $I/I^2 \otimes_Q Q = \Omega^1_{Q/A} \otimes_Q E$, and
for $a \in A, e \in E$ we have
$$
\begin{aligned}
\nabla_\sigma(a e) & = \sigma(ae) - 1 \otimes ae  \\
& = a \sigma(e) -1 \otimes ae \\
& = a\sigma(e) - a \otimes e + a \otimes e - 1 \otimes ae \\
& = a (\sigma(e) - 1 \otimes e) + (a \otimes 1 - 1 \otimes a) \otimes e \\
& = a \nabla_\sigma(e) + da \otimes e, \\
\end{aligned}
$$
since $da$ is identified with  $a \otimes 1 - 1 \otimes a$
under $\Omega^1_{Q/A} \cong I/I^2$.
\end{proof}

\begin{lem} \label{lem:comm}
Suppose $\cE, \cE'$ are locally free coherent sheaves on $X$ and
$\nabla, \nabla'$ are connections for each relative to $p$. If $g: \cE
\to \cE'$ is a morphisms of coherent sheaves, then  the map
$$
\nabla' \circ g  - (\id \otimes g) \circ \nabla:
\cE \to \Omega^1_{X/S} \otimes_{\cO_X} \cE'
$$
is a morphism of coherent sheaves.
\end{lem}

\begin{proof} Given an open set $U$ and elements $f \in \Gamma(U,
  \cO_X), e \in \Gamma(U, \cE)$, the displayed map sends $f \cdot e
  \in \Gamma(U, \cE)$ to
$$
\begin{aligned}
\nabla'(g(fe)) - (\id \otimes g)(df \otimes e - f \nabla(e)
& =
\nabla'(fg(e)) - df \otimes g(e) - f (\id \otimes g)(\nabla(e)) \\
& =
df \otimes g(e) +
f\nabla'(g(e)) - df \otimes g(e) - f (\id \otimes g)(\nabla(e)) \\
& =
f\nabla'(g(e)) - f (\id \otimes g)(\nabla(e)) \\
&
=
f\left(\left(\nabla' \circ g  -  (\id \otimes g) \circ \nabla\right)(e) \right).
\end{aligned}
$$
\end{proof}

\begin{prop} \label{AppProp225} 
For a vector bundle $\cE$ on $X$, the function $\sigma \mapsto
\nabla_\sigma$ determines a bijection between the set of splittings of
the map $\pi$ in 
\eqref{EA2} as a map of left modules
and
the set of connections on $\cE$
  relative to $p$.
In particular, $\cE$ admits a connection relative to $p$ if and only
if  $\cAt_{X/S}(\cE)= 0$. 
\end{prop}

\begin{proof} From Lemma \ref{lem:comm} with $g$ being the identity map, the difference of two connections on $\cE$ 
is $\cO_X$-linear.
By choosing any one splitting $\sigma_0$ of \eqref{EA2} and its associated
  connection $\nabla_0 = \nabla_{\sigma_0}$, the inverse of 
$$
\sigma \mapsto \nabla_{\sigma}
$$
is given by 
$$
\nabla \mapsto \nabla - \nabla_{\sigma_0} + \sigma_0.
$$
\end{proof}

\bibliographystyle{amsplain}

\providecommand{\bysame}{\leavevmode\hbox to3em{\hrulefill}\thinspace}
\providecommand{\MR}{\relax\ifhmode\unskip\space\fi MR }
\providecommand{\MRhref}[2]{%
  \href{http://www.ams.org/mathscinet-getitem?mr=#1}{#2}
}
\providecommand{\href}[2]{#2}

\end{document}